\documentclass[leqno]{amsart}
\usepackage{amsmath, amsthm, amscd, amsfonts, amssymb, graphicx, color, mathrsfs, stmaryrd}
\usepackage[bookmarksnumbered, colorlinks, plainpages]{hyperref}
\hypersetup{colorlinks=true,linkcolor=red, anchorcolor=green, citecolor=cyan, urlcolor=red, filecolor=magenta, pdftoolbar=true}

\newtheorem{thm}{Theorem}[section]

\newtheorem{pro}[thm]{Proposition}

\theoremstyle{definition}
\newtheorem{defn}[thm]{Definition}

\newtheorem{hyp}[thm]{Hypothesis}

\newtheorem{remark}[thm]{Remark}

\numberwithin{equation}{section}

\newcommand{\ol}[1]{\overline{#1}}

\renewcommand{\hat}[1]{\widehat{#1}}
\renewcommand{\tilde}[1]{\widetilde{#1}}

\newcommand{\set}[1]{{\left\{#1\right\}}}
\newcommand{\pa}[1]{{\left(#1\right)}}

\newcommand{\gen}[1]{{\left\langle #1\right\rangle}}
\newcommand{\abs}[1]{{\left|#1\right|}}
\newcommand{\norm}[1]{{\left\|#1\right\|}}

\newcommand{\ssm}{\smallsetminus}

\newcommand{\ra}{\rightarrow}

\newcommand{\N}{\mathbb{N}}

\newcommand{\R}{\mathbb{R}}

\newcommand{\E}{\mathbb{E}}

\newcommand{\eqsys}[1]{{\left\{\begin{array}{ll}#1\end{array}\right.}}

\newcommand{\elle}{\operatorname{L}}

\newcommand{\tc}{\, \middle |\,}                                                

\newcommand{\con}{\operatorname{\mathscr{C}}}

\newcommand{\eps}{\varepsilon}

\newcommand{\fcon}{\operatorname{\mathscr{FC}}}

\newcommand{\OU}{{\operatorname{\mathscr{L}}}}

\DeclareMathOperator{\lip}{\operatorname{Lip}}
\DeclareMathOperator{\diver}{\operatorname{div}}

\DeclareMathOperator{\grad}{\operatorname{grad}}

\makeatletter
\def\thebibliography#1{\section*{References\@mkboth
{References}{References}}\list
{[\arabic{enumi}]}{\settowidth\labelwidth{[#1]}\leftmargin\labelwidth
\advance\leftmargin\labelsep
\usecounter{enumi}}
\def\newblock{\hskip .11em plus .33em minus .07em}
\sloppy\clubpenalty4000\widowpenalty4000 
}
\makeatother

\begin{document}

\frenchspacing

\setcounter{page}{1}

\title[Sobolev Regularity for Infinite Dimensional Weighted Elliptic Problems]{Maximal Sobolev regularity for solutions of elliptic equations in infinite dimensional Banach spaces endowed with a weigthed Gaussian measure}

\author[G. Cappa and S. Ferrari]{{G. Cappa and {S. Ferrari$^1$}$^{*}$}}

\address{$^{1}$ Dipartimento di Matematica e Informatica, Universit\`a degli Studi di Parma, Parco Area delle Scienze 53/A, 43124 Parma, Italy.}
\email{\textcolor[rgb]{0.00,0.00,0.84}{gianluca.cappa@nemo.unipr.it, simone.ferrari1@unipr.it}}




\date{\today
\newline \indent $^{*}$ Corresponding author}

\begin{abstract}
Let $X$ be a separable Banach space endowed with a non-degenerate centered Gaussian measure $\mu$. The associated Cameron--Martin space is denoted by $H$.  Let $\nu=e^{-U}\mu$, where $U:X\ra\R$ is a sufficiently regular convex and continuous function. In this paper we are interested in the $W^{2,2}$ regularity of the weak solutions of elliptic equations of the type
\[\lambda u-L_\nu u=f,\]
where $\lambda>0$, $f\in \elle^2(X,\nu)$ and $L_\nu$ is the self-adjoint operator associated with the quadratic form
\[(\psi,\varphi)\mapsto \int_X\gen{\nabla_H\psi,\nabla_H\varphi}_Hd\nu\qquad\psi,\varphi\in W^{1,2}(X,\nu).\]
\end{abstract}

\maketitle

\section{Introduction} \label{Introduction}

Let $X$ be a separable Banach space with norm $\norm{\cdot}_X$, endowed with a non-degenerate centered Gaussian measure $\mu$. The associated Cameron--Martin space is denoted by $H$, its inner product by $\gen{\cdot,\cdot}_H$ and its norm by $\abs{\cdot}_H$.
The spaces $W^{1,p}(X,\mu)$ and $W^{2,p}(X,\mu)$ for $p\geq 1$ are the classical Sobolev spaces of the Malliavin calculus (see \cite[Chapter 5]{Bog98}).

The aim of this paper is to study the solutions of the equation
\begin{gather}
\lambda u-L_\nu u=f\label{Problema}
\end{gather}
where $\lambda >0$, $\nu$ is a measure of the form $e^{-U}\mu$ with $U:X\ra\R$ a convex and continuous function, $f\in\elle^2(X,\nu)$ and $L_{\nu}$ is the operator associated to the quadratic form
\[(\psi,\varphi)\mapsto \int_X\gen{\nabla_H\psi,\nabla_H\varphi}_Hd\nu\qquad\psi,\varphi\in W^{1,2}(X,\nu),\]
where $\nabla_H \psi$ represent the gradient along $H$ of $\psi$ and $W^{1,2}(\Omega,\nu)$ is the Sobolev space on $\Omega$ associate to the measure $\nu$ (see Section \ref{Notations and preliminaries}).

We need to clarify what we mean with \emph{solution} of problem \eqref{Problema}. We say that $u\in W^{1,2}(X,\nu)$ is a \emph{weak solution} of equation \eqref{Problema} if
\[\lambda\int_X u\varphi d\nu+\int_X\gen{\nabla_H u,\nabla_H\varphi}_Hd\nu=\int_Xf\varphi d\nu\qquad\text{for every }\varphi\in\fcon^\infty_b(X).\]
Notice that the weak solution is just $R(\lambda,L_\nu)f$, the resolvent of $L_\nu$.

In the finite dimensional case, existence, uniqueness and maximal regularity of the solution of equation \eqref{Problema} have been widely studied. Indeed in the case of the standard Gaussian measure in $\R^n$, the operator $L_\nu$ reads as
\[L_\nu u(\xi)=\Delta u(\xi)-\gen{\grad U(\xi)+\xi,\grad u(\xi)}\qquad u\in\con^2_b(\R^n),\]
so that, if $U$ is smooth, $L_\nu$ is an elliptic operator with smooth, although possibly unbounded, coefficients.
See for example \cite{DPG01}, \cite{BL07} and \cite{MPRS02}. In the infinite dimensional case maximal $W^{2,2}$ regularity results are known when $X$ is a separable Hilbert space. See for example \cite{DPL14}, where $U$ is assumed to be bounded from below. In the general Banach spaces case some results are known about equation \eqref{Problema}, but we do not know any $W^{2,2}$ regularity result. See for example \cite{AR91}, where a much larger class of operator is studied.

In order to state the results of this paper we need some hypotheses on the weighted measure $\nu$.

\begin{hyp}\label{ipotesi peso}
$U:X\ra\R$ is a convex and continuous function belonging to $W^{1,t}(X,\mu)$ for some $t>3$.  We set $\nu:=e^{-U}\mu$.
\end{hyp}
\noindent
The assumption $t>3$ may sound strange, but it is needed to define the weighted Sobolev spaces $W^{1,2}(X,\nu)$. Indeed observe that if $U$ satisfies Hypothesis \ref{ipotesi peso}, then it satisfies \cite[Hypothesis 1.1]{Fer15} since, by \cite[Lemma 7.5]{AB06}, $e^{-U}$ belongs to $W^{1,r}(X,\mu)$ for every $r<t$. Then following \cite{Fer15} it is possible to define the space $W^{1,2}(X,\nu)$ as the domain of the closure of the gradient operator along $H$.

The main result of this paper is the following theorem.

\begin{thm}\label{Main Theorem}
Let $U$ be a function satisfying Hypothesis \ref{ipotesi peso}, let $\lambda>0$ and $f\in\elle^2(X,\nu)$. Then equation (\ref{Problema}) has a unique weak solution $u\in W^{2,2}(X,\nu)$. Moreover $u$ satisfies
\begin{gather}
\norm{u}_{\elle^2(X,\nu)}\leq\frac{1}{\lambda}\norm{f}_{\elle^2(X,\nu)};\qquad \norm{\nabla_H u}_{\elle^2(X,\nu;H)}\leq\frac{1}{\sqrt{\lambda}}\norm{f}_{\elle^2(X,\nu)};\\
\norm{\nabla_H^2 u}_{\elle^2(X,\nu;\mathcal{H}_2)}\leq \sqrt{2}\norm{f}_{\elle^2(X,\nu)}.
\end{gather}
where $\nabla_H^2$ is defined in Section \ref{Notations and preliminaries} and $\mathcal{H}_2$ is the space of the Hilbert--Schmidt operators in $H$.
\end{thm}

The paper is organized in the following way: in section \ref{Notations and preliminaries} we recall some basic definitions and we fix the notations. Section \ref{H-Moreau envelop} is dedicated to modify a standard tool in the theory of convex functions on Hilbert spaces: the Moreau--Yosida approximations (see \cite{BC11} and \cite{Bre73}). In Section \ref{Finite dimension properties} we recall known results about finite dimensional elliptic and parabolic equations that we will use. In Section \ref{The nabla_H U H-Lipschitz case} we study the case in which $\nabla_H U$ is a $H$-Lipschitz function. Then we prove that equation (\ref{Problema}) admits a strong solution in the following sense:

\begin{defn}\label{definizione strong solution}
A function $u\in\elle^2(X,\nu)$ is a \emph{strong solution of equation (\ref{Problema})} if there exists a sequence $\set{u_n}_{n\in\N}\subseteq \fcon^3_b(X)$ such that $u_n$ converges to $u$ in $\elle^2(X,\nu)$ and
\[\elle^2(X,\nu)\text{-}\lim_{n\ra+\infty}\lambda u_n-L_\nu u_n=f.\]
Moreover a sequence $\set{u_n}_{n\in\N}\subseteq\fcon^3_b(X)$ satisfying the above conditions is called a \emph{strong solution sequence for $u$}.
\end{defn}
\noindent
We conclude the section proving Theorem \ref{Main Theorem}. In Section \ref{A charactetization of the domain of Lnu} we will recall some results about the divergence operator on weighted Gaussian spaces and we will show that if $U$ satisfies Hypothesis \ref{ipotesi peso} and $\nabla_H U$ is $H$-Lipschitz, then
$D(L_\nu)= W^{2,2}(X,\nu)$ and
\[\norm{u}_{D(L_\nu)}\leq\norm{u}_{W^{2,2}(X,\nu)}\leq \pa{2+\sqrt{2}}\norm{u}_{D(L_\nu)},\]
where $\norm{\cdot}_{D(L_\nu)}$ is the graph norm in $D(L_\nu)$, i.e. for $u\in D(L_\nu)$
\begin{gather}\label{formula norma grafico}
\norm{u}_{D(L_\nu)}:=\norm{u}_{\elle^2(X,\nu)}+\norm{L_\nu u}_{\elle^2(X,\nu)}.
\end{gather}
See Section \ref{Notations and preliminaries} for the definition of $W^{2,2}(X,\nu)$.
In the final section we show how our results can be applied to some examples. In our examples $X$ will be $\con_0[0,1]=\set{f\in\con[0,1]\tc f(0)=0}$, endowed with the classical Wiener measure $P^W$. First we consider, for $f\in\con_0[0,1]$,
\[U(f)=\int_0^1 f^2(\xi)d\xi\]
and we show that $U$ is a weight bounded from below, satisfies Hypothesis \ref{ipotesi peso} and $\nabla_H U$ is $H$-Lipschitz. In Example \ref{A unbuonded weight} we consider the following function, for $f\in\con_0[0,1]$,
\[U(f)=F(f)+f(1)\]
where $F(f)=\max_{\xi\in[0,1]}f(\xi)$. We show that $U$ satisfies Hypothesis \ref{ipotesi peso} although it is unbounded, both from above and from below.

\section{Notations and preliminaries} \label{Notations and preliminaries}

We will denote by $X^*$ the topological dual of $X$. We recall that $X^*\subseteq\elle^2(X,\mu)$. The linear operator $R_\mu:X^*\ra X^{**}$ defined by the formula
\[R_\mu x^*(y^*)=\int_X x^*(x)y^*(x)d\mu(x)\]
is called the covariance operator of $\mu$. We denote by $X^*_\mu$ the closure of $X^*$ in $\elle^2(X,\mu)$. The covariance operator $R_\mu$ can be extended by continuity to the space $X^*_\mu$. By \cite[Lemma 2.4.1]{Bog98} for every $h\in H$ there exists a unique $g\in X^*_\mu$ with $h= R_\mu(g)$, in this case we set
\begin{gather}\label{definizione hat}
\hat{h}:=g.
\end{gather}

Throughout the paper we fix an orthonormal basis $\set{e_i}_{i\in\N}$ of $H$ such that $\hat{e}_i$ belongs to $X^*$, for every $i\in\N$. Such basis exists by \cite[Corollary 3.2.8(ii)]{Bog98}.

We say that a function $f:X\ra\R$ is \emph{differentiable along $H$ at $x$} if there is $v\in H$ such that
\[\lim_{t\ra 0}\frac{f(x+th)-f(x)}{t}=\gen{v,h}_H\qquad\text{ uniformly for }h\in H\text{ with }\abs{h}_H=1.\]
In this case the vector $v\in H$ is unique and we set $\nabla_H f(x):=v$, moreover for every $k\in\N$ the derivative of $f$ in the direction of $e_k$ exists and it is given by
\begin{gather}\label{partial derivative definition}
\partial_k f(x):=\lim_{t\ra 0}\frac{f(x+te_k)-f(x)}{t}=\gen{\nabla_H f(x),e_k}_H.
\end{gather}

We denote by $\mathcal{H}_2$ the space of the Hilbert--Schmidt operators in $H$, that is the space of the bounded linear operators $A:H\ra H$ such that $\norm{A}_{\mathcal{H}_2}^2=\sum_{i}\abs{Ae_i}^2_H$ is finite (see \cite{DU77}).
We say that a function $f:X\ra\R$ is \emph{two times differentiable along $H$ at $x$} if it is differentiable along $H$ at $x$ and $A\in\mathcal{H}_2$ exists such that
\[H\text{-}\lim_{t\ra 0}\frac{\nabla_Hf(x+th)-\nabla_Hf(x)}{t}=A h\qquad\text{ uniformly for }h\in H\text{ with }\abs{h}_H=1.\]
In this case the operator $A$ is unique and we set $\nabla_H^2 f(x):=A$. Moreover for every $i,j\in\N$ we set
\begin{gather}\label{partial derivative definition 2}
\partial_{ij} f(x):=\lim_{t\ra 0}\frac{\partial_jf(x+te_i)-\partial_jf(x)}{t}=\langle\nabla_H^2 f(x)e_j,e_i\rangle_H.
\end{gather}

For $k\in\N\cup\set{\infty}$, we denote by $\fcon^k_b(X)$ the space of the cylindrical function of the type
\(f(x)=\varphi(x^*_1(x),\ldots,x^*_n(x))\)
where $\varphi\in\con^{k}_b(\R^n)$ and $x^*_1,\ldots,x^*_n\in X^*$ and $n\in\N$. We remark that $\fcon^\infty_b(X)$ is dense in $\elle^p(X,\nu)$ for all $p\geq 1$ (see \cite[Proposition 3.6]{Fer15}). We recall that if $f\in \fcon^2_b(X)$, then $\partial_{ij}f(x)=\partial_{ji}f(x)$ for every $i,j\in\N$ and $x\in X$.

The Gaussian Sobolev spaces $W^{1,p}(X,\mu)$ and $W^{2,p}(X,\mu)$, with $p\geq 1$, are the completions of the \emph{smooth cylindrical functions} $\fcon_b^\infty(X)$ in the norms
\begin{gather*}
\norm{f}_{W^{1,p}(X,\mu)}:=\norm{f}_{\elle^p(X,\mu)}+\pa{\int_X\abs{\nabla_H f(x)}_H^pd\mu(x)}^{\frac{1}{p}};\\
\norm{f}_{W^{2,p}(X,\mu)}:=\norm{f}_{W^{1,p}(X,\mu)}+\pa{\int_X\norm{\nabla_H^2 f(x)}^p_{\mathcal{H}_2}d\mu(x)}^{\frac{1}{p}}.
\end{gather*}
Such spaces can be identified with subspaces of $\elle^p(X,\mu)$ and the (generalized) gradient and Hessian along $H$, $\nabla_H f$ and $\nabla_H^2 f$, are well defined and belong to $\elle^p(X,\mu;H)$ and $\elle^p(X,\mu;\mathcal{H}_2)$, respectively. For more informations see \cite[Section 5.2]{Bog98}.

Now we consider $\nabla_H:\fcon^\infty_b(X)\ra\elle^p(X,\nu;H)$. This operator is closable in the norm of $\elle^p(X,\nu)$ whenever $p> \frac{t-1}{t-2}$ and Hypothesis \ref{ipotesi peso} holds (see \cite[Definition 4.3]{Fer15}). For such $p$ we denote by $W^{1,p}(X,\nu)$ the domain of its closure in $\elle^p(X,\nu)$.

Assume Hypotesis \ref{ipotesi peso} holds. We shall use the integration by parts formula (see \cite[Lemma 4.1]{Fer15}) for $\varphi\in W^{1,p}(X,\mu)$ with $p>\frac{t-1}{t-2}$:
\begin{gather}\label{int by part}
\int_X\partial_k\varphi d\nu=\int_X\varphi(\partial_kU+\hat{e}_k)d\nu\qquad\text{ for every }k\in\N,
\end{gather}
where $\hat{e}_k$ is defined in formula \eqref{definizione hat}.

In order to define the spaces $W^{2,p}(X,\nu)$, we need to prove the closability of the operator $(\nabla_H,\nabla_H^2)$ in $\elle^p(X,\nu)$.
\begin{pro}\label{chiusura hessiano}
Assume Hypothesis \ref{ipotesi peso} holds. For every $p\geq\frac{t-1}{t-2}$, the operator 
\[(\nabla_H,\nabla^2_H):\fcon^\infty_b(X)\ra \elle^p(X,\nu;H)\times\elle^p(X,\nu;\mathcal{H}_2)\] 
is closable in $\elle^p(X,\nu)$. The closure will be still denoted by $(\nabla_H,\nabla^2_H)$.
\end{pro}

\begin{proof}
Let $\set{\varphi_k}_{k\in\N}\subseteq \fcon^\infty_b(X)$ be such that $\varphi_k\ra 0$ in $\elle^p(X,\nu)$, $\nabla_H\varphi_k\ra F$ in $\elle^p(X,\nu;H)$, and $\nabla^2_H\varphi_k\ra\Phi$ in $\elle^p(X,\nu;\mathcal{H}_2)$ as $k\ra+\infty$. By \cite[Proposition 4.2]{Fer15} $F=0$ $\nu$-a.e.
Let $\psi\in\fcon^\infty_b(X)$, then by the integration by parts formula (formula \eqref{int by part}) we get
\[\int_X \psi\partial_{ij}\varphi_k d\nu=\int_X\psi\partial_j\varphi_k(\partial_i U+\hat{e}_i) d\nu-\int_X \partial_j\psi\partial_i\varphi_k d\nu.\]
We remark that
\begin{gather*}
\lim_{k\ra+\infty}\int_X \psi\partial_{ij}\varphi_k d\nu=\int_X \psi\gen{\Phi e_i,e_j}_H d\nu.\\
\intertext{Moreover}
\lim_{k\ra+\infty}\int_X \partial_j\psi\partial_i\varphi_k d\nu=0\qquad\text{and} \qquad\lim_{k\ra+\infty}\int_X\hat{e}_i\partial_j\varphi_k\psi d\nu=0.
\end{gather*}
Then $\int_X \psi\gen{\Phi e_i,e_j}_H d\nu=0$ for every $\psi\in\fcon^\infty_b(X)$.
So $\Phi = 0$ $\mu$-a.e. since $\fcon_b^\infty(X)$ is dense in $\elle^{p'}(X,\nu)$ (see \cite[Proposition 3.6]{Fer15}).
\end{proof}
\noindent
We are now able to define the Sobolev spaces $W^{2,p}(X,\nu)$.

\begin{defn}\label{definizione spazi di sobolev pesati}
Assume Hypothesis \ref{ipotesi peso} holds. For $p>\frac{t-1}{t-2}$ we denote by $W^{2,p}(X,\nu)$ the domain of closure of the operator $(\nabla_H,\nabla^2_H):\fcon_b^\infty(X)\ra \elle^p(X,\nu;H)\times\elle^p(X,\nu;\mathcal{H}_2)$ in $\elle^p(X,\nu)$.
\end{defn}
\noindent 
We remark that if $t>3$, i.e. when Hypothesis \ref{ipotesi peso} holds, then $2>\frac{t-1}{t-2}$.

We remind the reader that, if $U$ satisfies Hypothesis \ref{ipotesi peso} and belongs $W^{2,t}(X,\mu)$, where $t$ is the same as in Hypothesis \ref{ipotesi peso}, then by the integration by parts formula (formula \eqref{int by part}), and \cite[Proposition 5.3]{Fer15} (see also Proposition \ref{proposizione divergenza}) we get for every $u\in\fcon^2_b(X)$
\begin{gather}\label{formula Lnu}
L_\nu u=\sum_{i=1}^{+\infty}\partial_{ii}u-\sum_{i=1}^{+\infty}\pa{\partial_i U+\hat{e}_i}\partial_iu,
\end{gather}
where the series converges in $\elle^2(X,\nu)$.

Finally we recall the following corollary of the Hahn--Banach theorem (see \cite[Lemma 7.5]{AB06}).

\begin{pro}\label{minorante affine}
Let $g:X\ra\R\cup\set{+\infty}$ be a convex and lower semicontinuous function and let $r\in\R$ such that there is $x_0\in X$ with $g(x_0)>r$. Then there exists $x^*\in X^*$ such that for every $x\in X$
\[g(x)\geq x^*(x-x_0)+r.\]
\end{pro}

\section{Moreau--Yosida approximations along $H$} \label{H-Moreau envelop}

In this section we will modify a classical tool in the theory of convex functions in Hilbert spaces: the Moreau--Yosida approximations. For a classical treatment of the Moreau--Yosida approximations in Hilbert spaces we refer to \cite[Section 12.4]{BC11}.

Throughout this section $f:X\ra\R\cup\set{+\infty}$ is a convex and $\norm{\cdot}_X$-lower semicontinuous function.
Let $\alpha>0$ and define the \emph{Moreau--Yosida approximation along $H$} as
\begin{gather}\label{M env}
f_\alpha(x)=\inf\set{f(x+h)+\frac{1}{2\alpha}\abs{h}^2_H\tc h\in H}.
\end{gather}

\begin{pro}\label{proprieta g}
Let $x\in X$ and $\alpha>0$. The function $g_{\alpha,x}:H\ra\R$ defined as
\[g_{\alpha,x}(h)=f(x+h)+\frac{1}{2\alpha}\abs{h}^2_H,\]
is convex, $\abs{\cdot}_H$-lower semicontinuous and it has a unique global minimum point $P(x,\alpha)\in H$. Moreover
$g_{\alpha,x}$ is coercive, i.e.
\[\lim_{\abs{h}_H\ra +\infty}g_{\alpha, x}(h)=+\infty.\]
\end{pro}

\begin{proof}
Convexity is trivial. Let $H$-$\lim_{n\ra+\infty}h_n=h$. Since $H$ is continuously embedded in $X$, $X$-$\lim_{n\ra+\infty} h_n=h$. By the fact that $f$ is $\norm{\cdot}_X$-lower semicontinuous, we get
\[f(x+h)\leq\liminf_{n\ra+\infty}f(x+h_n).\]
So $g_{\alpha,x}$ is $\abs{\cdot}_H$-lower semicontinuous. By Proposition \ref{minorante affine}, for every $x\in X$ there exist $h(x)\in H$ and $\eta\in\R$ such that $f(x+h)\geq \gen{h,h(x)}_H+\eta$ for every $h\in H$. So we get
\begin{gather*}
\lim_{\abs{h}_H\ra+\infty}g_{\alpha,x}(h)=\lim_{\abs{h}_H\ra+\infty}\pa{f(x+h)+\frac{1}{2\alpha}\abs{h}_H^2}\geq\lim_{\abs{h}_H\ra+\infty}\pa{\gen{h,h(x)}_H+\eta+\frac{1}{2\alpha}\abs{h}_H^2}\geq\\
\geq \lim_{\abs{h}_H\ra+\infty}\pa{-\abs{h}_H\abs{h(x)}_H+\eta+\frac{1}{2\alpha}\abs{h}_H^2}=+\infty.
\end{gather*}
Since $g_{\alpha,x}$ is convex, $\abs{\cdot}_H$-lower semicontinuous and coercive,  the set
\[A_{\alpha, x}:=\set{p\in H\tc g_{\alpha,x}(p)=\inf \set{g_{\alpha,x}(h)\tc h\in H}}\]
is nonempty (see \cite[Proposition 11.14]{BC11}). We claim that $A_{\alpha,x}$ is a singleton. Indeed, by contradiction, assume that $p_1,p_2\in A_{\alpha,x}$ are such that $p_1\neq p_2$. Using the strict convexity of $\abs{\cdot}_H$
\begin{gather*}
g_{\alpha,x}\pa{\frac{p_1+p_2}{2}}=f\pa{x+\frac{p_1+p_2}{2}}+\frac{1}{2\alpha}\abs{\frac{p_1+p_2}{2}}^2_H<
\\
<\frac{1}{2}\pa{f(x+p_1)+\frac{1}{2\alpha}\abs{p_1}^2_H}+\frac{1}{2}\pa{f(x+p_2)+\frac{1}{2\alpha}\abs{p_2}^2_H}\leq\frac{1}{2}g_{\alpha,x}(p_1)+\frac{1}{2}g_{\alpha,x}(p_2)=\\
=\inf \set{g_{\alpha,x}(h)\tc h\in H},
\end{gather*}
a contradiction.
\end{proof}

\begin{pro}\label{Convergenza MY}
For every $x\in X$ we have $f_\alpha(x)\nearrow f(x)$ as $\alpha\ra 0^+$. In particular $f_\alpha(x)\leq f(x)$ for every $\alpha>0$ and $x\in X$.
\end{pro}

\begin{proof}
Monotonicity of $f_\alpha$ is obvious. Let $S(x):=\lim_{\alpha\ra 0^+}f_\alpha(x)=\sup_{\alpha\in(0,1)}f_\alpha(x)$. Since $f_\alpha(x)\leq f(x)$ we have $S(x)\leq f(x)$. If $S(x)=+\infty$ then there is nothing to prove.

Assume $S(x)<+\infty$. We just need to prove that $S(x)\geq f(x)$. By monotonicity we get
\[\set{P(x,\alpha)\tc\alpha\in(0,1)}\subseteq\set{h\in H\tc g_{1,x}(h)\leq S(x)}.\]
By Proposition \ref{proprieta g} the set $\set{P(x,\alpha)\tc\alpha\in(0,1)}\subseteq H$ is bounded. Let
\[c(x)=\sup\set{\abs{P(x,\alpha)}_H\tc\alpha\in(0,1)}.\]
By Proposition \ref{minorante affine}, for every $x\in X$ there exist $h(x)\in H$ and $\eta\in\R$ such that $f(x+h)\geq\gen{h,h(x)}_H+\eta$ for every $h\in H$. Then, for every $\alpha\in (0,1)$, we have
\begin{gather*}
S(x)\geq f_\alpha(x)=f(x+P(x,\alpha))+\frac{1}{2\alpha}\abs{P(x,\alpha)}_H^2\geq \gen{P(x,\alpha),h(x)}_H+\eta+\frac{1}{2\alpha}\abs{P(x,\alpha)}_H^2\geq\\
\geq -\abs{P(x,\alpha)}_H\abs{h(x)}_H+\eta+\frac{1}{2\alpha}\abs{P(x,\alpha)}_H^2\geq -c(x)\abs{h(x)}_H+\eta+\frac{1}{2\alpha}\abs{P(x,\alpha)}_H^2.
\end{gather*}
Then $\abs{P(x,\alpha)}^2_H\leq 2\alpha(S(x)+c(x)\abs{h(x)}_H-\eta)$ and $\abs{P(x,\alpha)}_H\ra 0$ as $\alpha\ra 0^+$. Finally
\begin{gather*}
S(x)=\lim_{\alpha\ra 0^+}f_\alpha(x)=\lim_{\alpha\ra 0^+}f(x+P(x,\alpha))+\frac{1}{2\alpha}\abs{P(x,\alpha)}_H^2\geq\liminf_{\alpha\ra 0^+}f(x+P(x,\alpha))\geq f(x).
\end{gather*}
\end{proof}

\begin{pro}\label{caratterizzazione punto minimo}
For $x\in X$ and $\alpha>0$ let $P(x,\alpha)$ be the unique minimum point of the function $g_{\alpha,x}$, given by Proposition \ref{proprieta g}. For $p\in H$, we have $p=P(x,\alpha)$ if, and only if,
\begin{gather}\label{proprieta punto minimo}
f(x+p)\leq f(x+h)+\frac{1}{\alpha}\gen{p,h-p}_H,
\end{gather}
for every $h\in H$.
\end{pro}

\begin{proof}
Let $\beta\in (0,1)$ and $h\in H$. Consider $p_\beta=\beta h+(1-\beta)P(x,\alpha)$ and observe that
\begin{gather*}
f_\alpha(x)=f(x+P(x,\alpha))+\frac{1}{2\alpha}\abs{P(x,\alpha)}_H^2\leq f(x+p_\beta)+\frac{1}{2\alpha}\abs{p_\beta}_H^2\leq\\
\leq \beta f(x+h)+(1-\beta)f(x+P(x,\alpha))+\frac{\beta^2}{2\alpha}\abs{h}^2_H
+\frac{\beta(1-\beta)}{\alpha}\gen{P(x,\alpha),h}_H+\frac{(1-\beta)^2}{2\alpha}\abs{P(x,\alpha)}_H^2.
\end{gather*}
Thus
\begin{gather*}
\beta f(x+P(x,\alpha))\leq \beta f(x+h)+\frac{\beta^2}{2\alpha}\abs{h}^2_H+\frac{\beta(1-\beta)}{\alpha}\gen{P(x,\alpha),h}_H+\frac{\beta(\beta-2)}{2\alpha}\abs{P(x,\alpha)}_H^2.
\end{gather*}
Dividing by $\beta$ we get
\begin{gather*}
f(x+P(x,\alpha))\leq f(x+h)+\frac{\beta}{2\alpha}\abs{h}^2_H+\frac{1-\beta}{\alpha}\gen{P(x,\alpha),h}_H+\frac{\beta-2}{2\alpha}\abs{P(x,\alpha)}_H^2,
\end{gather*}
and letting $\beta\ra 0^+$ we get
\[f(x+P(x,\alpha))\leq f(x+h)+\frac{1}{\alpha}\gen{P(x,\alpha),h-P(x,\alpha)}_H.\]

Conversely, observe that if $p\in H$ satisfies inequality \eqref{proprieta punto minimo}, then for every $h\in H$ we have
\begin{gather*}
f(x+p)+\frac{1}{2\alpha}\abs{p}_H^2\leq f(x+h)+\frac{1}{\alpha}\gen{p,h-p}_H+\frac{1}{2\alpha}\abs{p}_H^2\leq\\
\leq f(x+h)+\frac{1}{\alpha}\gen{p,h-p}_H+\frac{1}{2\alpha}\abs{p}_H^2+\frac{1}{2\alpha}\abs{h-p}_H^2=f(x+h)+\frac{1}{2\alpha}\abs{h}_H^2.
\end{gather*}
\end{proof}

\begin{pro}\label{punto di minimo lip}
Let $x\in X$ and $\alpha>0$. The function $P_{x,\alpha}:H\ra H$ defined as $P_{x,\alpha}(h):=P(x+h,\alpha)$ is Lipschitz continuous, with Lipschitz constant less or equal than $1$.
\end{pro}

\begin{proof}
Let $\alpha>0$, $x\in X$ and $h\in H$. By Proposition \ref{caratterizzazione punto minimo} we get
\begin{gather*}
f(x+P(x,\alpha))\leq f(x+h+P(x+h,\alpha))+\frac{1}{\alpha}\gen{P(x,\alpha),h+P(x+h,\alpha)-P(x,\alpha)}_H;\\
f(x+h+P(x+h,\alpha))\leq f(x+P(x,\alpha))+\frac{1}{\alpha}\gen{P(x+h,\alpha),P(x,\alpha)-h-P(x+h,\alpha)}_H.
\end{gather*}
Summing these inequalities and multiplying by $\alpha$ we get
\begin{gather*}
0\leq \gen{P(x,\alpha),h+P(x+h,\alpha)-P(x,\alpha)}_H+\gen{P(x+h,\alpha),P(x,\alpha)-h-P(x+h,\alpha)}_H=\\
=-\abs{P(x+h,\alpha)-P(x,\alpha)}_H^2+\gen{P(x,\alpha)-P(x+h,\alpha),h}_H.
\end{gather*}
Using the Cauchy--Schwartz inequality we get
\[\abs{P(x+h,\alpha)-P(x,\alpha)}_H^2\leq \gen{P(x,\alpha)-P(x+h,\alpha),h}_H\leq \abs{P(x+h,\alpha)-P(x,\alpha)}_H\abs{h}_H.\]
So $\abs{P(x+h,\alpha)-P(x,\alpha)}_H\leq \abs{h}_H$.
\end{proof}

\begin{pro}\label{MY differentiable}
Let $\alpha>0$. $f_\alpha$ is differentiable along $H$ at every point $x\in X$. Moreover, for every $x\in X$, we have
\[\nabla_H f_\alpha(x)=-\frac{1}{\alpha}P(x,\alpha).\]
\end{pro}

\begin{proof}
By Proposition \ref{caratterizzazione punto minimo}, for every $\alpha>0$ and $h\in H$, we get
\begin{gather*}
f_\alpha(x+h)-f_\alpha(x)=\\
=f(x+h+P(x+h,\alpha))+\frac{1}{2\alpha}\abs{P(x+h,\alpha)}_H^2-f(x+P(x,\alpha))-\frac{1}{2\alpha}\abs{P(x,\alpha)}_H^2\geq\\
\geq \frac{1}{2\alpha}\abs{P(x+h,\alpha)}_H^2-\frac{1}{\alpha}\gen{P(x,\alpha),h+P(x+h,\alpha)-P(x,\alpha)}_H-\frac{1}{2\alpha}\abs{P(x,\alpha)}_H^2=\\
=\frac{1}{2\alpha}\abs{P(x+h,\alpha)-P(x,\alpha)}_H^2-\frac{1}{\alpha}\gen{P(x,\alpha),h}_H\geq -\frac{1}{\alpha}\gen{P(x,\alpha),h}_H.
\end{gather*}
In a similar way, for every $\alpha>0$ and $h\in H$, we have
\[f_\alpha(x+h)-f_\alpha(x)\leq -\frac{1}{\alpha}\gen{P(x+h,\alpha),h}_H.\]
Combining these inequalities and applying Proposition \ref{punto di minimo lip} we get, for every $\alpha>0$ and $h\in H$,
\begin{gather*}
0\leq f_\alpha(x+h)-f_\alpha(x)+\frac{1}{\alpha}\gen{P(x,\alpha),h}_H\leq \frac{1}{\alpha}\gen{P(x,\alpha)-P(x+h,\alpha),h}_H\leq \\
\leq \frac{1}{\alpha}\abs{P(x,\alpha)-P(x+h,\alpha)}_H\abs{h}_H\leq \frac{1}{\alpha}\abs{h}_H^2.
\end{gather*}
So, for every $\alpha>0$, $f_\alpha$ is differentiable along $H$ at every point $x\in X$ and $\nabla_H f_\alpha(x)=-\frac{1}{\alpha}P(x,\alpha)$.
\end{proof}

\begin{pro}\label{Moreau in W22}
Let $\alpha>0$ and $1\leq p<+\infty$. If $f\in\elle^p(X,\mu)$, then $f_\alpha \in W^{2,p}(X,\mu)$.
\end{pro}

\begin{proof}
By Proposition \ref{MY differentiable} we get, for every $\alpha>0$ and $x\in X$,
\[\nabla_H f_\alpha(x)=-\frac{1}{\alpha}P(x,\alpha).\]
Proposition \ref{punto di minimo lip} and \cite[Theorem 5.11.2]{Bog98} give us that for every $\alpha>0$, $\nabla_H f_\alpha\in W^{1,q}(X,\mu,H)$ for every $q\geq 1$. The conclusion follows from the inequality $f_\alpha(x)\leq f(x)$ for every $\alpha>0$ and $x\in X$ (Proposition \ref{Convergenza MY}).
\end{proof}

\section{Finite dimensional results} \label{Finite dimension properties}

In this section we recall some known finite dimensional results about the operator
\begin{gather}\label{formula L finito dim}
L_\phi \psi=\sum_{i=1}^nD_{ii}\psi-\sum_{i=1}^n (D_i\phi+\xi_i)D_i\psi,
\end{gather}
where $\phi$ is a convex function with Lipschitz continuous gradient, and $\psi\in \con^2_b(\R^n)$. We mainly refer to the results in \cite{BF04}.
We need a dimension-free uniform estimate for $u$ and $\grad u$, where $u$ is a solution of
\begin{gather}\label{prob_finito_ellittico}
\lambda u(\xi)- L_\phi u(\xi)= f(\xi)\qquad \xi\in\R^n,
\end{gather}
where $\lambda>0$ and $f$ is a bounded $\gamma$-H\"older continuous function, for some $0<\gamma<1$.
Recall that the space $\con_b^{k+\gamma}(\R^n)$, for $k\in\N\cup\set{0}$ and $0<\gamma< 1$, is the space of the $k$-differentiable functions with bounded and $\gamma$-H\"older derivatives up to the order $k$, endowed with its standard norm (see \cite[Section 2.7]{Tri95}), i.e. for $f\in\con_b^{k+\gamma}(\R^n)$ we let $\norm{f}_{\con_b^{k+\gamma}(\R^n)}=\norm{f}_{\con_b^k(\R^n)}+[D^kf]_\gamma$ where
\[[D^kf]_\gamma=\sum_{\abs{\beta}=k}\sup\set{\frac{\abs{D^{\beta}f(\xi_1)-D^{\beta}f(\xi_2)}}{\abs{\xi_1-\xi_2}^\gamma}\tc \xi_1,\xi_2\in\R^n,\ \xi_1\neq\xi_2}.\]
Also the space $\con^{k+\beta,m+\gamma}(A\times\R^n)$ for $k,m\in\N\cup \set{0}$, $0<\beta,\gamma<1$ and $A$ an open subset of $\R$ is the space of $k$-differentiable functions with $\beta$-H\"older derivatives up to the order $k$ in the first variable and $m$-differentiable functions with $\gamma$-H\"older derivatives up to the order $m$ in the second variable. As usual when we add the subscript \emph{loc} we mean that the H\"older condition holds locally.

The following result will be useful.

\begin{pro}\label{Teo holder}
Let $0<\gamma<1$ and assume that $\phi$ has a Lipschitz continuous gradient. For every $f\in\con_b^\gamma(\R^n)$ equation \eqref{prob_finito_ellittico} has a unique solution $u\in\con_b^{2+\gamma}(\R^n)$, and there exists a constant $C>0$, independent of $f$, such that
\begin{gather}\label{stima holder}
\norm{u}_{\con_b^{2+\gamma}(\R^n)}\leq C\norm{f}_{\con_b^\gamma(\R^n)}.
\end{gather}
Moreover if $\phi,f\in\con^\infty(\R^n)$, then $u\in\con^\infty(\R^n)$.
\end{pro}
\noindent Inequality \eqref{stima holder} was proved in \cite[Theorem 1]{LV98}, and the local regularity result can be found in \cite[Theorem 3.1.1]{LU68}.

Consider the problem
\begin{gather}
\left\{\begin{array}{ll}
D_t v(t,\xi)=L_\phi v(t,\xi) &  t>0,\ \xi\in\R^n;\\
v(0,\xi)=f(\xi), &  \xi\in\R^n.
\end{array}\right.\label{prob_finito_evol}
\end{gather}
$L_\phi$ satisfies the conditions (2.1), (2.2), (2.3), and (2.4) of \cite{BF04}.
By \cite[Theorem 3.1]{BF04}, for every $f\in \con_b(\R^n)$ there exists a unique bounded solution $v$ of problem \eqref{prob_finito_evol} belonging to $\con([0,\infty)\times\R^n)\cap \con^{1+\gamma/2,2+\gamma}_{\text{loc}}((0,\infty)\times\R^n)$. If we set
\begin{gather}\label{tt}
T_t f (\xi)= v(t, \xi),\qquad t\geq 0,\ \xi\in\R^n,
\end{gather}
then $\{T_t\}_{t\geq0}$ is a positive contraction semigroup on $\con_b(\R^n)$.

We want a dimension-free uniform estimate of the gradient of $T_tf$. Before proceeding we prove that the function $g(\xi)=\abs{\xi}^2$ satisfies
\begin{gather}\label{Lyapunov}
\lim_{\abs{\xi}\ra+\infty}g(\xi)=+\infty\qquad\text{ and }\qquad\sup_{\xi\in\R^n}L_\phi g(\xi)<+\infty.
\end{gather}
A function $g$ satisfying \eqref{Lyapunov} is said to be a Lyapunov function for the operator $L_\phi$. The first condition in \eqref{Lyapunov} is obviously satisfied. Moreover
\begin{gather*}
L_\phi g(\xi)=2n-2\gen{\grad\phi(\xi),\xi}-2\abs{\xi}^2 =2n-2\gen{\grad\phi(\xi)-\grad\phi(0),\xi}-2\gen{\grad\phi(0),\xi}-2\abs{\xi}^2\leq\\
\leq 2n+2\abs{\grad\phi(0)}\abs{\xi}-2\abs{\xi}^2=2n+2\abs{\grad\phi(0)}\abs{\xi}-2\abs{\xi}^2,
\end{gather*}
where we have used the fact that $\gen{\grad\phi(\xi)-\grad\phi(0),\xi}\geq 0$ for every $\xi\in\R^n$ since $\phi$ is a differentiable convex function (see \cite[Example 2.2(a)]{Phe93}). So the second condition in \eqref{Lyapunov} is satisfied. This implies that $g$ is a Lyapunov function and we get the following formulation of \cite[Proposition 2.1]{Lun98}.

\begin{pro}\label{principio massimo}
Assume that $\phi\in\con^\infty(\R^n)$ is convex and has Lipschitz continuous gradient. Let $T>0$ and let $z_0\in \con([0,T]\times \R^n)$ be a bounded function. Let $z\in\con([0,T]\times \R^n)\cap \con^{1,2}((0,T]\times \R^n)$ be a bounded function satisfying
\[\eqsys{D_t z(t,\xi)-L_{\phi}z(t,\xi)\leq 0, &\quad 0<t\leq T,\ \xi\in\R^n;\\
z(0,\xi)=z_0(\xi), &\quad \xi\in\R^n.}\]
If $\sup z> 0$, then
\[\sup_{\xi\in\R^n}z(t,\xi)\leq \sup_{\xi\in\R^n}z_0(\xi)\qquad 0\leq t\leq T.\]
\end{pro}
\noindent
The dimension-free uniform estimate of the gradient of $T_tf$ follows from an application of Bernstein's method, we give the proof just for the sake of completeness. More general results can be found in \cite{BF04}, \cite{BFL07}, \cite{BL05} and \cite{Lun98}, where larger classes of operators are studied, but no explicit dimension-free uniform estimates of the gradient of $T_tf$ are emphasized.

\begin{pro}\label{stima uniforme gradiente}
Assume that $\phi\in\con^\infty(\R^n)$ is convex and has Lipschitz continuous gradient. Then for every $t\geq 0$ and $\xi\in\R^n$ we have $\abs{T_tf(\xi)}\leq\norm{f}_\infty$ and
\[\abs{\grad T_t f(\xi)}\leq \frac{\norm{f}_\infty}{\sqrt{t}}\qquad t>0\text{ and } \xi\in\R^n,\]
for every $f\in\con^\infty_b(\R^n)$.
\end{pro}

\begin{proof}
If $f\equiv 0$ then the conclusion is obvious. So we can assume, without loss of generality, that $f\not\equiv 0$. We set
\begin{gather}\label{equazione zbeta}
z(t,\xi):=\abs{v(t,\xi)}^2+t\abs{\grad v(t,\xi)}^2\qquad t>0\text{ and }\xi\in\R^n
\end{gather}
where $v(t,\xi)=T_t f(\xi)$. From the general regularity theory of parabolic problems we get that $v$ is smooth for $t\geq 0$. We claim that the function $z$ satisfies the hypotheses of Proposition \ref{principio massimo}. Indeed
\begin{gather*}
D_t z(t,\xi)=2 v(t,\xi)D_t v(t,\xi)+\abs{\grad v(t,\xi)}^2+2 t\sum_{i=1}^nD_i v(t,\xi)D_i D_t v(t,\xi)=\\
=2 v(t,\xi)\Delta v(t,\xi)-2 v(t,\xi)\gen{\grad\phi(\xi)+\xi,\grad v(t,\xi)}
+\abs{\grad v(t,\xi)}^2+2 t\gen{\grad(\Delta v(t,\xi)),\grad v(t,\xi)}+\\
-2 t\sum_{i,j=1}^n \pa{D_j(D_i\phi+\xi_i)D_i v(t,\xi)D_j v(t,\xi)+(D_i\phi+\xi_i)D_{ij}v(t,\xi)D_j v(t,\xi)}.
\end{gather*}
Now we compute $L_\phi z$. We have
\begin{gather*}
L_{\phi}z(t,\xi)=2\abs{\grad v(t,\xi)}^2+2 v(t,\xi)\Delta v(t,\xi)+2 t\gen{\grad(\Delta v(t,\xi)),\grad v(t,\xi)}+\\
+2 t\sum_{i,j=1}^n(D_{ij}v(t,\xi))^2-2 v(t,\xi)\gen{\grad\phi(\xi)+\xi,\grad v(t,\xi)}-2 t\sum_{i,j=1}^n(D_i\phi+\xi_i)D_{ij}v(t,\xi)D_j v(t,\xi).
\end{gather*}
Then we get
\begin{gather*}
D_t z(t,\xi)-L_{\phi}z(t,\xi)
=-\abs{\grad v(t,\xi)}^2-2 t\sum_{i,j=1}^n(D_{ij}v(t,\xi))^2+\\
-2 t\gen{D^2\phi(\xi)\grad v(t,\xi),\grad v(t,\xi)}-2t\abs{\grad v(t,\xi)}^2.
\end{gather*}
Since $\phi$ is a convex function, $D^2\phi$ is positive-semidefinite matrix, and so
\[D_t z(t,\xi)-L_{\phi}z(t,\xi)\leq 0\qquad t>0,\ \xi\in\R^n.\]
Let $T>0$. Since $z(0,\xi)=(f(\xi))^2$, we can apply Proposition \ref{principio massimo} and we get
\[\sup_{\xi\in\R^n}z(t,\xi)\leq \norm{f}^2_\infty\qquad 0\leq t\leq T.\]
By equation \eqref{tt} and equation \eqref{equazione zbeta}
\[\abs{\grad T_tf(\xi)}\leq \frac{\norm{f}_\infty}{\sqrt{t}}\qquad 0< t\leq T,\ \xi\in\R^n.\]
Since the above estimate does not depend on $T$ we can conclude
\[\abs{\grad T_tf(\xi)}\leq \frac{\norm{f}_\infty}{\sqrt{t}}\qquad t>0,\ \xi\in\R^n.\]
In the same way we get $\abs{T_tf(\xi)}\leq\norm{f}_\infty$ for every $t\geq 0$ and $\xi\in\R^n$.
\end{proof}

By \cite[Proposition 3.2]{BF04} and \cite[Proposition 3.6]{Pri99} there exists an operator $A$ whose resolvent is
\begin{gather}\label{formula risolvente}
R(\lambda,A)f(\xi)=\int_0^{+\infty} e^{-\lambda t} (T_t f)(\xi)dt\qquad \xi\in\R^n.
\end{gather}
By \cite[Proposition 3.4]{BF04} if $\psi\in\con^2_b(\R^n)$, then $A\psi=L_\phi\psi$.

\begin{pro}\label{stime di shauder}
Assume that $\phi\in\con^\infty(\R^n)$ is convex and has Lipschitz continuous gradient. Let $u$ be a classical solution of equation \eqref{prob_finito_ellittico}. Then
\[\abs{\grad u(\xi)}\leq \sqrt{\frac{\pi}{\lambda}}\norm{f}_\infty\qquad \xi\in\R^n.\]
Furthermore $\norm{u}_\infty\leq\lambda^{-1}\norm{f}_\infty$.
\end{pro}

\begin{proof}
The furthermore part follows from the contractivity of $T_t$ and formula \eqref{formula risolvente}.
By Proposition \ref{stima uniforme gradiente} we can differentiate under the integral sign in formula \eqref{formula risolvente} and we get
\[\grad u(\xi)=\int_0^{+\infty} e^{-\lambda t}(\grad T_t f)(\xi)dt\qquad \xi\in\R^n.\]
Moreover, for every $\xi\in\R^n$
\begin{gather*}
\abs{\grad u(\xi)}\leq \int_0^{+\infty}\frac{e^{-\lambda t}}{\sqrt{t}}dt\norm{f}_\infty=\sqrt{\frac{\pi}{\lambda}}\norm{f}_\infty.
\end{gather*}
\end{proof}

\section{Passing to infinite dimension}\label{The nabla_H U H-Lipschitz case}

This section is devoted to prove Theorem \ref{Main Theorem}. We start by showing that if $\nabla_H U$ is $H$-Lipschitz, then equation \eqref{Problema} has a unique strong solution, in the sense of Definition \ref{definizione strong solution}, and this solution satisfies the Sobolev regularity estimates listed in Theorem \ref{Main Theorem}.

We need to recall some basic definitions that can be found in \cite{Bog98}. Let $Y$ be a separable Banach space, we recall that a function $F:X\ra Y$ is said to be $H$-Lipschitz if $C>0$ exists such that
\begin{gather}\label{Hlip}
\norm{F(x+h)-F(x)}_Y\leq C\abs{h}_H,
\end{gather}
for every $h\in H$ and $\mu$-a.e. $x\in X$ (see \cite[Section 4.5 and Section 5.11]{Bog98}). We denote by $P_n:X\ra H$ the projection
\[P_n(x)=\sum_{i=1}^n\hat{e}_i(x)e_i\qquad\text{for every }x\in X,\]
where $\hat{e}_i$ belongs to $X^*$, for every $i\in\N$ (formula \eqref{definizione hat}). Let \(\mu_n:=\mu\circ P_n^{-1} \text { and }\tilde{\mu}_n:=\mu\circ(I-P_n)^{-1}\). Recall that both measures are non-degenerate, centered and Gaussian on $P_nX$ and $(I-P_n)X$ respectively, and
\begin{gather}\label{cameron martin di proiezione}
\tilde{H}_n=(I-P_n)(H),
\end{gather}
is the Cameron--Martin space associated with the measure $\tilde{\mu}_n$ on $(I-P_n)X$. For the proofs of such results see \cite[Theorem 3.7.3]{Bog98}.

Let $f\in\elle^p(X,\mu)$ for some $p\geq 1$ and $n\in\N$. We denote by $\E_n f$ the conditional expectation of $f$, i.e. for every $x\in X$
\[\E_nf(x)=\int_Xf\pa{P_nx+(I-P_n)y}d\mu(y).\]
We recall in the following proposition the results in \cite[Corollary 3.5.2 and Proposition 5.4.5]{Bog98}

\begin{pro}\label{convergenza En f}
Let $1\leq p<+\infty$ and $f\in\elle^p(X,\mu)$. Then $\E_n f$ converges to $f$ in $\elle^p(X,\mu)$ and $\mu$-a.e. and for every $n\in\N$
\[\norm{\E_n f}_{\elle^p(X,\mu)}\leq\norm{f}_{\elle^p(X,\mu)}.\]
Moreover if $f\in W^{1,p}(X,\mu)$, then $\E_n f$ converges to $f$ in $W^{1,p}(X,\mu)$ and $\mu$-a.e., for every $n\in\N$ we have $\norm{\E_n f}_{W^{1,p(X,\mu)}}\leq\norm{f}_{W^{1,p}(X,\mu)}$ and
\[\partial_i\E_n f=\eqsys{\E_n\partial_i f & 1\leq i\leq n;\\
0 & i>n.}\]
Finally the same results, with obvious modifications, are true if $f\in W^{2,p}(X,\mu)$.
\end{pro}

\subsection{The case where $\nabla_H U$ is $H$-Lipschitz}

In this subsection we will assume the following hypothesis on the weight:

\begin{hyp}\label{ipotesi peso 2}
Let $U:X\ra \R$ be a function satisfying Hypothesis \ref{ipotesi peso}. Assume that $U$ is differentiable along $H$ at every point $x\in X$, and $\nabla_H U$ is $H$-Lipschitz. We will denote by $[\nabla_H U]_{H\text{-}\lip}$ the $H$-Lipschitz constant of $\nabla_H U$, i.e. the constant $C$ in formula \eqref{Hlip}.
\end{hyp}
\noindent
We recall that by \cite[Theorem 5.11.2]{Bog98} we have $U\in W^{2,t}(X,\mu)$, where $t$ is the same as in Hypothesis \ref{ipotesi peso}. Observe that every convex function in $\fcon^2_b(X)$ and every continuous linear functional $x^*\in X^*$ satisfy Hypothesis \ref{ipotesi peso 2}.

Let $f\in \fcon_b^\infty(X)$ be such that $f(x)=\varphi(\hat{e}_1(x),\ldots,\hat{e}_{N_0}(x))$ for some $N_0\in\N$ and $\varphi\in\con_b^\infty(\R^{N_0})$. Throughout the rest of this subsection we let $n> N_0$.

\begin{pro}\label{psi n}
Consider the function $\psi_n:\R^n\ra \R$ defined as
\[\psi_n(\xi):=\int_X U\pa{\sum_{i=1}^n\xi_i e_i +(I-P_n) y}d\mu(y),\]
where $\xi=(\xi_1,\ldots,\xi_n)\in\R^n$. Then $\psi_n$ belongs to $\con^1(\R^n)$ and it has Lipschitz gradient with Lipschitz constant less or equal than the $H$-Lipschitz constant of $\nabla_H U$.
\end{pro}

\begin{proof}
Let $\set{d_i\tc\ i=1,\ldots,n}$ be the canonical basis of $\R^n$. We will prove that $\psi_n$ admits derivative along $d_i$ for every $i=1,\ldots,n$ and that the gradient is Lipschitz continuous. This implies that
$\psi_n$ is continuous.

First of all we prove that for every $i\in\N$ the function $\E_n\partial_i U(x)$ is finite everywhere. For every $x\in X$
\begin{gather*}
\abs{\E_n\partial_iU(x)}\leq\int_X\abs{\partial_iU\pa{P_nx+(I-P_n)y}}d\mu(y)\leq\\\leq \int_X\abs{\partial_iU\pa{P_nx+(I-P_n)y}}^2d\mu(y)\leq
\int_X\abs{\nabla_HU\pa{P_nx+(I-P_n)y}}_H^2d\mu(y)\leq\\
\leq 2\int_X\abs{\nabla_HU\pa{P_nx+(I-P_n)y}-\nabla_HU(y-P_ny)}_H^2d\mu(y)+2\int_X\abs{\nabla_HU(y-P_ny)}_H^2d\mu(y)\leq\\
\leq 2[\nabla_HU]_{H\text{-}\lip}^2\int_X\abs{P_nx}_H^2d\mu(y)+2\int_X\abs{\nabla_H U(y-P_ny)}^2_Hd\mu(y)=\\
=2[\nabla_HU]_{H\text{-}\lip}^2\abs{P_nx}_H^2+2\int_{(I-P_n)X}\abs{\nabla_H U(z)}^2_Hd\tilde{\mu}_n(z).
\end{gather*}
The last term of this chain of inequalities is finite, indeed $\nabla_HU$ is $\tilde{H}_n$-Lipschitz continuous, by formula \eqref{cameron martin di proiezione}, and the conclusion follows from \cite[Theorem 5.11.2]{Bog98}.

Since $U$ is continuous and for every $x\in X$ the function $U$ is differentiable along $H$ at $x$, with $H$-Lipschitz gradient along $H$, then for every $x\in X$ and $h\in H$, the function
\(F_{x,h}(t):=U(x+th)\)
belongs to $\con^1[0,1]$. Indeed $F_{x,h}'(t)=\gen{\nabla_H U(x+th),h}_H$ and for every $t_1,t_2\in (0,1)$ we have
\(|F_{x,h}'(t_1)-F_{x,h}'(t_2)|\leq [\nabla_HU]_{H\text{-}\lip}\abs{h}_H^2\abs{t_1-t_2}\). So $F_{x,h}'$ is Lipschitz continuous. So by the fundamental theorem of calculus we get
\begin{gather}\label{teo fond calcolo}
U(x+h)-U(x)=\int_0^1\gen{\nabla_H U(x+th),h}_H dt.
\end{gather}
Now we get for every $i=1,\ldots,n$ and $s\in (0,1)$
\begin{gather*}
\abs{\frac{\psi_n(\xi+sd_i)-\psi_n(\xi)}{s}-\E_n \partial_i U\pa{\sum_{j=1}^n\xi_je_j}}
=\left|\frac{1}{s}\int_X U\pa{\sum_{j=1}^n\pa{\xi_j +s\delta_{ij}}e_j +(I-P_n) y}\right.+\\
\left.-U\pa{\sum_{j=1}^n\xi_je_j +(I-P_n) y}d\mu(y)-\E_n \partial_i U\pa{\sum_{j=1}^n\xi_je_j}\right|=\\
=\left|\frac{1}{s}\int_X\int_0^1\gen{\nabla_H U\pa{\sum_{j=1}^n\xi_je_j +(I-P_n) y+st\cdot e_i},s\cdot e_i}_H dt d\mu(y)-\E_n \partial_i U\pa{\sum_{j=1}^n\xi_je_j}\right|=\\
=\left|\frac{1}{s}\int_X\int_0^1\left\langle\nabla_H U\pa{\sum_{j=1}^n\xi_je_j +(I-P_n) y +st\cdot e_i}
-\nabla_H U\pa{\sum_{j=1}^n\xi_je_j +(I-P_n) y},s\cdot e_i\right\rangle_H dt d\mu(y)\right|\leq\\
\leq \abs{s}[\nabla_H U]_{H\text{-}\lip}\int_X\int_0^1t dtd\mu(y)=\frac{\abs{s}}{2}[\nabla_H U]_{H\text{-}\lip},
\end{gather*}
which goes to zero as $s\ra 0$. So $D_i\psi_n(\xi)=(\E_n\partial_i U)(\sum_{j=1}^n\xi_je_j)$. Finally, for every $\xi,\eta\in \R^n$
\begin{gather*}
\abs{\grad\psi_n(\xi)-\grad\psi_n(\eta)}^2=\sum_{i=1}^n\abs{D_i\psi_n(\xi)-D_i\psi_n(\eta)}^2=\\
=\sum_{i=1}^n\abs{\int_X\partial_i U\pa{\sum_{j=1}^n\xi_ie_i+(I-P_n)y}-\partial_i U\pa{\sum_{j=1}^n\eta_ie_i+(I-P_n)y}d\mu(y)}^2\leq\\
\leq \int_X\abs{\nabla_H U\pa{\sum_{j=1}^n\xi_ie_i+(I-P_n)y}-\nabla_H U\pa{\sum_{j=1}^n\eta_ie_i+(I-P_n)y}}_H^2d\mu(y)\leq\\
\leq [\nabla_H U]_{H\text{-}\lip}^2\int_X\sum_{i=1}^n(\xi_i-\eta_i)^2d\mu(y)= [\nabla_H U]_{H\text{-}\lip}^2\abs{\xi-\eta}^2.
\end{gather*}
So we have $\abs{\grad\psi_n(\xi)-\grad\psi_n(\eta)}\leq [\nabla_H U]_{H\text{-}\lip}\abs{\xi-\eta}$, for every $\xi,\eta\in\R^n$.
\end{proof}

Now we mollify the functions $\psi_n$. Fix $\eps>0$ and $\theta\in\con^\infty_b(\R^n)$ with support contained in the unit ball and $\int_{\R^n}\theta(\xi)d\xi=1$. Let
\[\psi_n^\eps(\xi)=\int_{\R^n}\psi_n(\xi-\eps\eta)\theta(\eta)d\eta.\]
Then $\psi_n^\eps$ is convex, it belongs to $\con^{\infty}_b(\R^n)$ and $\grad \psi_n^\eps$ is Lipschitz continuous. For $\lambda>0$ consider the problem
\begin{gather}\label{equazione finito dimensional}
\lambda v^\eps_n(\xi)-\OU^{(n,\eps)}_\nu v_n^\eps(\xi)=\varphi(\pi_{N_0}\xi),\qquad \xi\in\R^n,
\end{gather}
where $\pi_{N_0}:\R^n\ra\R^{N_0}$ is the projection on the first $N_0$ coordinates, and $\OU^{(n,\eps)}_\nu$ is the following operator:
\begin{gather}\label{eq Lcorsivo_nu^n}
\OU^{(n,\eps)}_\nu v=\sum_{i=1}^nD_{ii}v-\sum_{i=1}^n(D_i\psi_n^\eps+\xi_i) D_iv,\qquad v\in \con_b^2(\R^n).
\end{gather}
By Proposition \ref{Teo holder} we know that equation (\ref{equazione finito dimensional}) admits a unique solution $v_n^\eps$ belonging to $\con^{\infty}(\R^n)\cap\bigcup_{\gamma\in(0,1)}\con^{2+\gamma}_b(\R^n)$.

\begin{pro}\label{Soluzion e' c3b}
$v_n^\eps$ belongs to $\con^3_b(\R^n)$.
\end{pro}

\begin{proof}
We just need to prove that the third order derivatives are bounded. We start by differentiating equation (\ref{equazione finito dimensional}),
\begin{gather*}
(1+\lambda)D_jv_n^\eps(\xi)-\OU^{(n,\eps)}_\nu D_jv_n^\eps(\xi)=D_j\varphi(\pi_{N_0}\xi)-\sum_{i=1}^nD_jD_i\psi_n^\eps(\xi)D_iv_n^\eps(\xi)\qquad \text{ for }1\leq j\leq N_0;\\
(1+\lambda)D_jv_n^\eps(\xi)-\OU^{(n,\eps)}_\nu D_jv_n^\eps(\xi)=-\sum_{i=1}^nD_jD_i\psi_n^\eps(\xi)D_iv_n^\eps(\xi)\qquad \text{ for }j=N_0+1,\ldots, n.
\end{gather*}
In both equations the right hand side is Lipschitz continuous and bounded. By Proposition \ref{Teo holder} we get $D_jv_n^\eps\in\bigcup_{\gamma\in(0,1)}\con^{2+\gamma}_b(\R^n)$ for every $j=1,\ldots,n$. In particular $v_n^\eps$ belongs to $\con^3_b(\R^n)$.
\end{proof}

We return to infinite dimension. Set
\[U_n^\eps(x):=\psi_n^\eps(\hat{e}_1(x),\ldots,\hat{e}_n(x)),\qquad V_n^\eps(x):=v_n^\eps(\hat{e}_1(x),\ldots, \hat{e}_n(x)),\qquad x\in X.\]
Let $\nu_{n}^\eps=e^{-U_{n}^\eps}\mu$. The operator $L^{(n,\eps)}_{\nu}$ is defined as $L$, namely
\begin{gather*}
D(L_\nu^{(n,\eps)})=\bigg\{u\in W^{1,2}(X,\nu_n^{\eps})\,|\, \text{there exists }g\in\elle^2(X,\nu_n^\eps)\text{ such that }\\
\left.\int_X\gen{\nabla_H u,\nabla_H \varrho}_Hd\nu_n^\eps=-\int_Xg\varrho d\nu_n^\eps\text{ for every }\varrho\in\fcon^\infty_b(X)\right\},
\end{gather*}
and if $u\in D(L_\nu^{(n,\eps)})$ we let $L_\nu^{(n,\eps)}u=g$. It is easily seen that $\fcon^2_b(X)\subseteq D(L_\nu^{(n,\eps)})$ for every $n\in\N$ and $\eps>0$. Furthermore if $Z\in\fcon^2_b(X)$ is such that $Z(x)=\omega(\hat{e}_1(x),\ldots,\hat{e}_k(x))$ for some $k\in\N$ and $\omega\in\con^2_b(\R^k)$, then
\begin{gather}\label{eq L_nu^n}
L_\nu^{(n,\eps)} Z=\sum_{i=1}^k\partial_{ii}Z-\sum_{i=1}^k(\partial_i U_n^\eps+\hat{e}_i) \partial_iZ.
\end{gather}

\begin{pro}
Assume Hypothesis \ref{ipotesi peso 2} holds. The function $V_n^\eps\in D(L_\nu)\cap D(L_\nu^{(n,\eps)})$. For every $x\in X$
\begin{gather}
\lambda V_n^\eps-L_\nu V_n^\eps=f+\gen{\nabla_H U-\nabla_HU_n^\eps,\nabla_H V_n^\eps}_H;\label{Equazione soddisfatta da perturbazione}\\
\label{equazione inutile}L_\nu^{(n,\eps)}V_n^\eps(x)=\OU_\nu^{(n,\eps)}v_n^\eps(\hat{e}_1(x),\ldots,\hat{e}_n(x));
\end{gather}
Moreover the following inequality holds for every $x\in X$
\begin{gather}
\abs{\nabla_H V_n^\eps(x)}_H\leq \sqrt{\frac{\pi}{\lambda}}\norm{\varphi}_\infty.\label{Stime di schauder infinito dimensionali}
\end{gather}
\end{pro}

\begin{proof}
Since $V_n^\eps\in\fcon^3_b(X)$, $V_n^\eps\in D(L_\nu)\cap D(L_\nu^{(n,\eps)})$. Equality \eqref{Equazione soddisfatta da perturbazione} and equality \eqref{equazione inutile} follow from equality (\ref{eq Lcorsivo_nu^n}) and equality (\ref{eq L_nu^n}) and some computations. The moreover part is a consequence of Proposition \ref{stime di shauder}.
\end{proof}

\begin{pro}\label{Convergenza V_n^eps}
Assume Hypothesis \ref{ipotesi peso 2} holds. Then $\lambda V_n^{\frac{1}{n}}-L_\nu V_n^{\frac{1}{n}}$ converges to $f$ in $\elle^2(X,\nu)$ as $n$ goes to $+\infty$.
\end{pro}

\begin{proof}
Using equality (\ref{Equazione soddisfatta da perturbazione}) and inequality (\ref{Stime di schauder infinito dimensionali}) we get
\begin{gather*}
\int_X\abs{\lambda V_n^{\frac{1}{n}}(x)-L_\nu V_n^{\frac{1}{n}}(x)-f(x)}^2d\nu(x)=\int_X\abs{\gen{\nabla_H U(x)-\nabla_HU_n^{\frac{1}{n}}(x),\nabla_H V_n^{\frac{1}{n}}(x)}_H}^2d\nu(x)\leq\\
\leq \frac{\pi}{\lambda}\norm{\varphi}^2_\infty\int_X\abs{\nabla_H U(x)-\nabla_H U_n^{\frac{1}{n}}(x)}_H^2d\nu(x)\leq\\
\leq \frac{\pi}{\lambda}\norm{\varphi}^2_\infty\pa{\int_X\abs{\nabla_H U(x)-\nabla_H \E_nU(x)}_H^2d\nu(x)+\int_X\abs{\nabla_H \E_nU(x)-\nabla_H U_n^{\frac{1}{n}}(x)}_H^2d\nu(x)}.
\end{gather*}
We recall that due to Hypothesis \ref{ipotesi peso 2} we have $t>3$ and $e^{-U}$ belongs to $\elle^{\frac{t}{t-2}}(X,\mu)$ (see the discussion after Hypothesis \ref{ipotesi peso}). Then
\[\int_X\abs{\nabla_H U(x)-\nabla_H \E_nU(x)}_H^2d\nu(x)\leq \norm{e^{-U}}_{\elle^{\frac{t}{t-2}}(X,\mu)}\pa{\int_X\abs{\nabla_H U(x)-\nabla_H \E_nU(x)}_H^td\mu(x)}^{\frac{2}{t}};\]
by Proposition \ref{convergenza En f} the integral in the right hand side vanishes as $n\ra+\infty$.

Let $\mu_n=\mu\circ P_n^{-1}$ and let $[\grad \psi_n]_1$ be the Lipschitz constant of $\grad\psi_n$. By the change of variable formula (see \cite[Formula (A.3.1)]{Bog98}) and Proposition \ref{psi n} we get
\begin{gather*}
\int_X\abs{\nabla_H \E_nU(x)-\nabla_H U_n^{\frac{1}{n}}(x)}_H^2d\nu(x)\leq\\
\leq \norm{e^{-U}}_{\elle^{\frac{t}{t-2}}(X,\mu)}\pa{\int_{\R^n}\abs{\grad\psi_n(\xi)-\grad\psi_n^{\frac{1}{n}}(\xi)}^{t}d\mu_n(\xi)}^{\frac{2}{t}}\leq\\
\leq \norm{e^{-U}}_{\elle^{\frac{t}{t-2}}(X,\mu)}\pa{\int_{\R^n}\pa{\int_{\R^n}\abs{\grad\psi_n(\xi)-\grad\psi_n(\xi-n^{-1}\eta)}\theta(\eta)d\eta}^{t}d\mu_n(\xi)}^{\frac{2}{t}}\leq\\
\leq \norm{e^{-U}}_{\elle^{\frac{t}{t-2}}(X,\mu)}\pa{\frac{[\grad \psi_n]_{1}}{n}}^{2}\pa{\int_{\R^n}\abs{\eta}\theta(\eta)d\eta}^{2}\leq \norm{e^{-U}}_{\elle^{\frac{t}{t-2}}(X,\mu)}\pa{\frac{[\nabla_H U]_{H\text{-}\lip}}{n}}^{2}.
\end{gather*}
The last term of this chain of inequalities goes to zero as $n\ra+\infty$
\end{proof}

\begin{pro}\label{Densita}
Assume Hypothesis \ref{ipotesi peso 2} holds. Then $(\lambda I-L_\nu)(\fcon^3_b(X))$ is dense in $\elle^2(X,\nu)$.
\end{pro}

\begin{proof}
It follows from Proposition \ref{Convergenza V_n^eps} and the density of the space $\fcon^\infty_b(X)$ in $\elle^2(X,\nu)$ (see \cite[Proposition 3.6]{Fer15}).
\end{proof}

\begin{pro}\label{existence of strong solution}
Assume Hypothesis \ref{ipotesi peso 2} holds. For every $\lambda >0$ and $f\in\elle^2(X,\nu)$, there exists a unique strong solution of equation (\ref{Problema}) in the sense of Definition \ref{definizione strong solution}.
\end{pro}

\begin{proof}
First of all observe that $L_\nu:\fcon^3_b(X)\ra\elle^2(X,\nu)$ is a dissipative operator. Indeed, for every $u\in\fcon^3_b(X)$, we have
\[\int_X uL_\nu ud\nu\leq 0.\]
Combining Proposition \ref{Densita} and the Lumer--Phillips theorem (see \cite[Theorem 2.3.15]{EN06}), we get that the closure $\ol{L_\nu}$ of the operator $L_\nu$ generates a contraction semigroup and $\fcon^3_b(X)$ is a core for $\ol{L_\nu}$, i.e. it is dense in $D(\ol{L_\nu})$ with the graph norm. In particular for every $\lambda>0$ and $f\in\elle^2(X,\nu)$, equation (\ref{Problema}) has a unique strong solution $u\in D(\ol{L_\nu})$.
\end{proof}

We recall the following theorem (see \cite[Theorem 3.1(2)]{FU00}).

\begin{thm}\label{Hessiano e convessita}
Let $F\in W^{2,p}(X,\mu)$, for some $p>1$, be a convex function. Then $\nabla_H^2F$ is a positive Hilbert--Schmidt operator $\mu$-a.e., i.e. \(\langle\nabla_H^2F(x)h,h\rangle_H\geq 0\), for $\mu$-a.e. $x\in X$ and every $h\in H$.
\end{thm}

We will state now a regularity result when $U$ satisfies Hypothesis \ref{ipotesi peso 2}.

\begin{thm}\label{Stime per lip}
Let $U$ be a function satisfying Hypothesis \ref{ipotesi peso 2}. Let $\lambda>0$, $f\in\elle^2(X,\nu)$, and let $u$ be the strong solution of equation (\ref{Problema}). Then $u\in W^{2,2}(X,\nu)$ and
\begin{gather}
\label{1 stime max per lip}\norm{u}_{\elle^2(X,\nu)}\leq\frac{1}{\lambda}\norm{f}_{\elle^2(X,\nu)};\qquad \norm{\nabla_H u}_{\elle^2(X,\nu;H)}\leq\frac{1}{\sqrt{\lambda}}\norm{f}_{\elle^2(X,\nu)};\\
\label{2 stime max per lip}\norm{\nabla_H^2 u}_{\elle^2(X,\nu;\mathcal{H}_2)}\leq \sqrt{2}\norm{f}_{\elle^2(X,\nu)}.
\end{gather}
Moreover $u$ is a weak solution of equation (\ref{Problema}). Finally if $\set{u_n}_{n\in\N}\subseteq\fcon^3_b(X)$ is a strong solution sequence for $u$ (see Definition \ref{definizione strong solution}), then $u_n$ converges to $u$ in $W^{2,2}(X,\nu)$.
\end{thm}

\begin{proof}
By Proposition \ref{existence of strong solution} a sequence $\set{u_n}_{n\in\N}\subseteq \fcon^3_b(X)$ and a function $u\in D(\ol{L_\nu})\subseteq W^{1,2}(X,\nu)$ exist such that $u_n$ converges to $u$ in $\elle^2(X,\nu)$ and
\[\elle^2(X,\nu)\text{-}\lim_{n\ra+\infty}\lambda u_n-L_\nu u_n=f.\]
Let $f_n:=\lambda u_n-L_\nu u_n$. By formula (\ref{formula Lnu}) and the fact that $U\in W^{2,t}(X,\mu)$, where $t$ is the same as in Hypothesis \ref{ipotesi peso}, we get $f_n\in W^{1,2}(X,\nu)$. Multiplying by $u_n$ and integrating we get
\begin{gather*}
\int_Xf_n(x)u_n(x)d\nu(x)=\lambda\int_Xu_n^2(x)d\nu(x)-\int_Xu_n(x)L_\nu u_n(x)(x)d\nu(x)=\\
=\lambda\int_Xu_n^2(x)d\nu(x)+\int_X\abs{\nabla_H u_n(x)}^2_Hd\nu(x).
\end{gather*}
Using the Cauchy--Schwarz inequality in the left hand side integral we get
\begin{gather}\label{label del referee}
\norm{u_n}_{\elle^2(X,\nu)}\leq\frac{1}{\lambda}\norm{f_n}_{\elle^2(X,\nu)};\qquad \norm{\nabla_H u_n}_{\elle^2(X,\nu;H)}\leq\frac{1}{\sqrt{\lambda}}\norm{f_n}_{\elle^2(X,\nu)}.
\end{gather}
Since $\set{u_n}_{n\in\N}$ and $\set{f_n}_{n\in\N}$ converge to $u$ and $f$, respectively, in $\elle^2(X,\nu)$ we get
\[\norm{u}_{\elle^2(X,\nu)}=\lim_{n\ra+\infty}\norm{u_n}_{\elle^2(X,\nu)}\leq\lim_{n\ra+\infty}\frac{1}{\lambda} \norm{f_n}_{\elle^2(X,\nu)}=\frac{1}{\lambda}\norm{f}_{\elle^2(X,\nu)}.\]
Moreover
\[\norm{\nabla_H u_n-\nabla_H u_m}_{\elle^2(X,\nu;H)}\leq\frac{1}{\sqrt{\lambda}}\norm{f_n-f_m}_{\elle^2(X,\nu)},\]
then $\{\nabla_H u_n\}_{n\in\N}$ is a Cauchy sequence in $\elle^2(X,\nu;H)$. By the closability of $\nabla_H$ in $\elle^2(X,\nu)$ it follows that $u\in W^{1,2}(X,\nu)$ and
\[\elle^2(X,\nu;H)\text{-}\lim_{n\ra+\infty}\nabla_H u_n=\nabla_H u.\]
Therefore
\[\norm{\nabla_H u}_{\elle^2(X,\nu;H)}=\lim_{n\ra+\infty}\norm{\nabla_H u_n}_{\elle^2(X,\nu;H)}\leq\lim_{n\ra+\infty}\frac{1}{\sqrt{\lambda}} \norm{f_n}_{\elle^2(X,\nu)}=\frac{1}{\sqrt{\lambda}}\norm{f}_{\elle^2(X,\nu)}.\]

Using formula \eqref{formula Lnu}, we differentiate the equality $\lambda u_n-L_\nu u_n=f_n$ with respect to the $e_j$ direction, we multiply the result by $\partial_j u$, sum over $j$ and finally integrate over $X$ with respect to $\nu$. We obtain
\begin{gather*}
(1+\lambda)\int_X\abs{\nabla_H u_n}_H^2d\nu+\int_X\norm{\nabla_H^2 u_n}_{\mathcal{H}_2}^2d\nu+\int_X\gen{\nabla_H^2 U\nabla_H u_n,\nabla_H u_n}_Hd\nu
=\int_Xf_n^2d\nu-\lambda\int_Xf_nu_nd\nu.
\end{gather*}
Using inequalities \eqref{label del referee} and Theorem \ref{Hessiano e convessita} we get
\[\norm{\nabla_H^2 u_n}_{\elle^2(X,\nu;\mathcal{H}_2)}\leq \sqrt{2}\norm{f_n}_{\elle^2(X,\nu)}.\]
We remark that
\[\norm{\nabla^2_H u_n-\nabla^2_H u_m}_{\elle^2(X,\nu;\mathcal{H}_2)}\leq\sqrt{2}\norm{f_n-f_m}_{\elle^2(X,\nu)},\]
then $\{\nabla^2_H u_n\}_{n\in\N}$ is a Cauchy sequence in $\elle^2(X,\nu;\mathcal{H}_2)$. By the closability of $\nabla^2_H$ in $\elle^2(X,\nu)$ it follows that $u\in W^{2,2}(X,\nu)$ and
\[\elle^2(X,\nu;\mathcal{H}_2)\text{-}\lim_{n\ra+\infty}\nabla_H^2 u_n=\nabla_H^2 u.\]
Therefore
\[\norm{\nabla^2_H u}_{\elle^2(X,\nu;\mathcal{H}_2)}=\lim_{n\ra+\infty}\norm{\nabla^2_H u_n}_{\elle^2(X,\nu;\mathcal{H}_2)}\leq\lim_{n\ra+\infty}\sqrt{2} \norm{f_n}_{\elle^2(X,\nu)}=\sqrt{2}\norm{f}_{\elle^2(X,\nu)},\]
and $\set{u_n}_{n\in\N}$ converges to $u$ in $W^{2,2}(X,\nu)$. 

Now we want to show that $u$ is a weak solution of equation (\ref{Problema}). Let $\varphi\in\fcon_b^\infty(X)$ and $n\in\N$, then
\[\lambda\int_X u_n \varphi d\nu-\int_X L_\nu u_n\varphi d\nu=\int_X f_n\varphi d\nu.\]
By the definition of $L_\nu$, we get
\begin{gather}
\lambda\int_X u_n \varphi d\nu+\int_X \gen{\nabla_H u_n,\nabla_H\varphi}_H d\nu=\int_X f_n\varphi d\nu.
\label{formulazione debole u_n}
\end{gather}
Since $\set{u_n}_{n\in\N}$ converges to $u$ in $W^{2,2}(X,\nu)$ we obtain
\[\lim_{n\ra+\infty}\lambda\int_X u_n \varphi d\nu=\lambda\int_X u \varphi d\nu,\]
and
\[\lim_{n\ra+\infty}\int_X \gen{\nabla_H u_n,\nabla_H\varphi}_H d\nu=\int_X \gen{\nabla_H u,\nabla_H\varphi}_H d\nu.\]
Since $\set{f_n}_{n\in\N}$ converges to $f$ in $\elle^2(X,\nu)$ we have
\[\lim_{n\ra+\infty}\int_X f_n \varphi d\nu=\int_X f \varphi d\nu.\]
Then taking the limit as $n$ goes to $+\infty$ in \eqref{formulazione debole u_n} we get that $u$ is a weak solution of equation \eqref{Problema}, i.e. for every $\varphi\in\fcon^\infty_b(X)$, we have
\(\lambda\int_X u \varphi d\nu+\int_X \gen{\nabla_H u,\nabla_H\varphi}_H d\nu=\int_X f\varphi d\nu\).
\end{proof}

\begin{remark}
The hypothesis of continuity of the function $U$, in Hypothesis \ref{ipotesi peso 2}, can be replaced by the weaker hypothesis of $H$-continuity, i.e. for $\mu$-a.e. $x\in X$
\[\lim_{H\ni h\ra 0}U(x+h)=U(x).\]
Anyway we will use the results of this section for the Moreau--Yosida approximations along $H$ of a function $U$ satisfying Hypothesis \ref{ipotesi peso}, that are continuous in our case.
\end{remark}

\subsection{The general case}

In this subsection we assume that $U$  satisfies Hypothesis \ref{ipotesi peso}. In this case we do not know if there exists a strong solution of equation \eqref{Problema}, but the Lax--Milgram theorem gives us a weak solution of equation \eqref{Problema}.

Let $\alpha\in(0,1]$ and let $U_\alpha$ be the Moreau--Yosida approximation along $H$ of $U$, defined in Section \ref{H-Moreau envelop}. Consider the measure
\begin{gather}
\nu_\alpha=e^{-U_\alpha}\mu.
\end{gather}

\begin{pro}\label{Integrabilita MY}
Let $\alpha\in(0,1]$. $U_\alpha$ satisfies Hypothesis \ref{ipotesi peso 2}. Moreover $e^{-U_\alpha}\in W^{1,p}(X,\mu)$, for every $p\geq 1$, and $U_\alpha\in W^{2,t}(X,\mu)$, where $t$ is given by Hypothesis \ref{ipotesi peso}.
\end{pro}

\begin{proof}
By Proposition \ref{minorante affine} there exist $x^*\in X^*$ and $\eta\in\R$ such that $U_1(x)\geq x^*(x)+\eta$ for every $x\in X$. Then by Proposition \ref{Convergenza MY}, for every $x\in X$ we have
\begin{gather*}
U_\alpha(x)\geq U_1(x)\geq x^*(x)+\eta.
\end{gather*}
So $e^{-U_\alpha(x)}\leq e^{-x^*(x)-\eta}$ for every $x\in X$. By the change of variable formula (see \cite[Formula (A.3.1)]{Bog98}) we obtain
\begin{gather*}
\int_X e^{-x^*(x)-\eta}d\mu(x)=e^{-\eta}\int_\R e^{-\xi}d\mu_{x^*}(\xi)<+\infty,
\end{gather*}
where $\mu_{x^*}=\mu\circ (x^{*})^{-1}$. So $e^{-U_\alpha}\in \elle^{p}(X,\mu)$  for every $p\geq 1$. By the differentiability of $U_\alpha$ along $H$ (see Proposition \ref{MY differentiable}) we get $\nabla_H e^{-U_\alpha(x)}=-e^{-U_\alpha(x)}\nabla_H U_\alpha(x)$ for every $x\in X$. By Proposition \ref{punto di minimo lip} and \cite[Theorem 5.11.2]{Bog98} we get $e^{-U_\alpha}\in W^{1,p}(X,\mu)$, for every $p\geq 1$. Finally $U_\alpha\in W^{2,t}(X,\mu)$, by Proposition \ref{Moreau in W22}.

Differentiability along $H$ and the $H$-Lipschitzianity of $\nabla_H U_\alpha$ follow from Proposition \ref{punto di minimo lip} and Proposition \ref{MY differentiable}. The convexity follows from the following standard argument: let $\eps>0$, $x_1,x_2\in X$ and $\lambda\in[0,1]$ and consider $h_\eps(x_1),h_\eps(x_2)\in H$ such that for $i=1,2$
\[U(x+h_\eps(x_i))+\frac{1}{2\alpha}\abs{h_\eps(x_i)}_H^2\leq U_\alpha(x_i)+\eps.\]
We get
\begin{gather*}
U_\alpha(\lambda x_1+(1-\lambda)x_2)\leq\\
\leq U(\lambda x_1+(1-\lambda)x_2+\lambda h_\eps(x_1)+(1-\lambda)h_\eps(x_2))+\frac{1}{2\alpha}\abs{\lambda h_\eps(x_1)+(1-\lambda)h_\eps(x_2)}^2_H\leq\\
\leq \lambda\pa{U(x_1+h_\eps(x_1))+\frac{1}{2\alpha}\abs{h_\eps(x_1)}_H^2}+(1-\lambda)\pa{U(x_2+h_\eps(x_2))+\frac{1}{2\alpha}\abs{h_\eps(x_2)}_H^2}\leq\\
\leq \lambda U_\alpha(x_1)+(1-\lambda)U_\alpha(x_2)+\eps.
\end{gather*}
Letting $\eps\ra 0$ we get the convexity of $U_\alpha$ for every $\alpha\in(0,1]$. Continuity of $U_\alpha$ is a consequence of Proposition \ref{Convergenza MY} and \cite[Corollary 2.4]{ET99}.
\end{proof}

Arguing as in \cite[Proposition 4.2]{Fer15} and using Proposition  \ref{Integrabilita MY} we get
\[\nabla_H:\fcon^\infty_b(X)\ra\elle^2(X,\nu_\alpha;H)\]
is a closable operator in $\elle^2(X,\nu_\alpha)$ for every $\alpha\in(0,1]$. The same is true for the operator $(\nabla_H,\nabla_H^2):\fcon^\infty_b(X)\ra\elle^2(X,\nu_\alpha;H)\times\elle^2(X,\nu_\alpha;\mathcal{H}_2)$ (see Proposition \ref{chiusura hessiano}). In particular we can define the spaces $W^{1,2}(X,\nu_\alpha)$ and $W^{2,2}(X,\nu_\alpha)$ as the domains of their respective closures. 

For $\alpha\in(0,1]$, consider now the operator
\begin{gather*}
D(L_{\nu_\alpha})=\bigg\{u\in W^{1,2}(X,\nu_\alpha)\,|\, \text{there exists }v\in\elle^2(X,\nu_\alpha)\text{ such that }\\
\left.\int_X\gen{\nabla_H u,\nabla_H \varphi}d\nu_\alpha=-\int_Xv\varphi d\nu_\alpha\text{ for every }\varphi\in\fcon^\infty_b(X)\right\},
\end{gather*}
with $L_{\nu_\alpha}u=v$ if $u\in D(L_{\nu_\alpha})$.

Now we have all the tools needed to prove Theorem \ref{Main Theorem}. The arguments are similar to those in \cite[Theorem 3.9]{DPL14}, we give the proof just for the sake of completeness.

\begin{proof}[Proof of Theorem \ref{Main Theorem}]
Let $\ol{f}\in\fcon_b^\infty(X)$ and $\set{\alpha_n}_{n\in\N}\subseteq(0,1]$ be a decreasing sequence converging to zero. Consider the family of equations
\begin{gather}
\lambda u_{\alpha_n}-L_{\nu_{\alpha_n}}u_{\alpha_n}=\ol{f}.\label{Problema con alpha}
\end{gather}
By Proposition \ref{existence of strong solution} and Theorem \ref{Stime per lip}, for every $n\in\N$, equation \eqref{Problema con alpha} has a unique strong solution, which coincides with the weak solution, $u_{\alpha_n}\in W^{2,2}(X,\nu_{\alpha_n})$ such that
\begin{gather}\label{stime max regol per alpha n}
\begin{array}{c}
\displaystyle\norm{u_{\alpha_n}}_{\elle^2(X,\nu_{\alpha_n})}\leq\frac{1}{\lambda}\norm{\ol{f}}_{\elle^2(X,\nu_{\alpha_n})};\qquad \norm{\nabla_H u_{\alpha_n}}_{\elle^2(X,\nu_{\alpha_n};H)}\leq\frac{1}{\sqrt{\lambda}}\norm{\ol{f}}_{\elle^2(X,\nu_{\alpha_n})};\\
\displaystyle\norm{\nabla_H^2 u_{\alpha_n}}_{\elle^2(X,\nu_{\alpha_n};\mathcal{H}_2)}\leq \sqrt{2}\norm{\ol{f}}_{\elle^2(X,\nu_{\alpha_n})}.
\end{array}
\end{gather}
By Proposition \ref{Convergenza MY} and Proposition \ref{Integrabilita MY} we get, for every $n\in\N$,
\begin{gather}\label{stima uniforme per f in l2(x nu alpha)}
\norm{\ol{f}}^2_{\elle^2(X,\nu_{\alpha_n})}\leq \int_X\abs{\ol{f}(x)}^2e^{-U_1(x)}d\mu(x)\leq\norm{\ol{f}}_\infty^2\int_Xe^{-U_1(x)}d\mu(x)<+\infty.
\end{gather}
By Proposition \ref{Convergenza MY} we have $e^{-U}\leq e^{-U_{\alpha_n}}$. So the set $\set{u_{\alpha_n}\tc n\in\N}$ is bounded in $W^{2,2}(X,\nu)$.

By weak compactness a function $u\in W^{2,2}(X,\nu)$ and a subsequence, which we still denote by $\set{u_{\alpha_n}}_{n\in\N}$, exist such that $u_{\alpha_n}\ra u$ weakly in $W^{2,2}(X,\nu)$ and $u_{\alpha_n}$, $\nabla_Hu_{\alpha_n}$, $\nabla_H^2u_{\alpha_n}$ converge pointwise $\mu$-a.e. respectively to $u$, $\nabla_H u$ and $\nabla_H^2 u$.

By inequality \eqref{stima uniforme per f in l2(x nu alpha)} and the Lebesgue dominated convergence theorem we get
\begin{gather}\label{convergenza f segnato}
\lim_{n\ra+\infty}\norm{\ol{f}}_{\elle^2(X,\nu_{\alpha_n})}=\norm{\ol{f}}_{\elle^2(X,\nu)}.
\end{gather}
By the weak convergence of $\{u_{\alpha_n}\}_{n\in\N}$ in $W^{2,2}(X,\nu)$ to $u$, the lower semicontinuity of the norm of $\elle^2(X,\nu)$, $\elle^2(X,\nu;H)$ and $\elle^2(X,\nu;\mathcal{H}_2)$, inequalities \eqref{stime max regol per alpha n} and equality \eqref{convergenza f segnato} we have
\begin{gather*}
\norm{u}_{\elle^2(X,\nu)}\leq\liminf_{n\ra\infty}\norm{u_{\alpha_n}}_{\elle^2(X,\nu)} \leq\liminf_{n\ra\infty}\norm{u_{\alpha_n}}_{\elle^2(X,\nu_{\alpha_n})}\leq\frac{1}{\lambda}\liminf_{n\ra\infty}\norm{\ol{f}}_{\elle^2(X,\nu_{\alpha_n})}=\frac{1}{\lambda}\norm{\ol{f}}_{\elle^2(X,\nu)};
\end{gather*}
\begin{gather*}
\norm{\nabla_Hu}_{\elle^2(X,\nu;H)}\leq\liminf_{n\ra\infty}\norm{\nabla_Hu_{\alpha_n}}_{\elle^2(X,\nu;H)} \leq\liminf_{n\ra\infty}\norm{\nabla_Hu_{\alpha_n}}_{\elle^2(X,\nu_{\alpha_n};H)}\leq\\ 
\leq\frac{1}{\sqrt{\lambda}}\liminf_{n\ra\infty}\norm{\ol{f}}_{\elle^2(X,\nu_{\alpha_n})}=\frac{1}{\sqrt{\lambda}}\norm{\ol{f}}_{\elle^2(X,\nu)};
\end{gather*}
and
\begin{gather*}
\norm{\nabla^2_Hu}_{\elle^2(X,\nu;\mathcal{H}_2)}\leq\liminf_{n\ra\infty}\norm{\nabla^2_Hu_{\alpha_n}}_{\elle^2(X,\nu;\mathcal{H}_2)} \leq\liminf_{n\ra\infty}\norm{\nabla^2_Hu_{\alpha_n}}_{\elle^2(X,\nu_{\alpha_n};\mathcal{H}_2)}\leq\\ \leq\sqrt{2}\liminf_{n\ra\infty}\norm{\ol{f}}_{\elle^2(X,\nu_{\alpha_n})}=\sqrt{2}\norm{\ol{f}}_{\elle^2(X,\nu)}.
\end{gather*}
Now we show that $u$ is a weak solution of the equation
\[\lambda u-L_\nu u=\ol{f}.\]
We recall that $\{u_{\alpha_n}\}_{n\in\N}$ is a sequence of weak solutions of the equations \eqref{Problema con alpha}, i.e.
\begin{gather}
\lambda\int_X  u_{\alpha_n}\varphi\, d\nu_{\alpha_n}+\int_X\gen{\nabla_H u_{\alpha_n},\nabla_H\varphi}_H d\nu_{\alpha_n}=\int_X \ol{f} \varphi\,d\nu_{\alpha_n}
\label{formulazione debole u_alpha_n}
\end{gather}
for all $\varphi\in\fcon_b^\infty(X)$ and $n\in\N$. By inequalities \eqref{stime max regol per alpha n}, for every $\varphi\in\fcon_b^\infty(X)$ and $n\in\N$, we have
\begin{gather*}
\int_X\abs{u_{\alpha_n}\varphi} e^{-U_{\alpha_n}}d\mu\leq \norm{\varphi}_{\infty}\int_X\abs{u_{\alpha_n}}e^{-U_{\alpha_n}}d\mu\leq \norm{\varphi}_{\infty}\left(\int_X u^2_{\alpha_n}d\nu_{\alpha_n}\right)^{\frac{1}{2}}\left(\int_X e^{-U_{\alpha_n}}d\mu\right)^{\frac{1}{2}}\leq\\
\leq\frac{1}{\lambda}\norm{\varphi}_{\infty}\norm{\ol{f}}_{\elle^2(X,\nu_{\alpha_n})}\left(\int_X e^{-U_{1}}d\mu\right)^{\frac{1}{2}}\leq \frac{1}{\lambda}\norm{\varphi}_{\infty}\norm{\ol{f}}_{\infty}\left(\int_X e^{-U_{1}}d\mu\right).
\end{gather*}
Then by Proposition \ref{Convergenza MY}, Proposition \ref{Integrabilita MY}, the pointwise $\mu$-a.e. convergence of $u_{\alpha_n}$ to $u$, and the Lebesgue dominated convergence theorem we get
\[\lim_{n\ra+\infty} \lambda\int_X u_{\alpha_n}\varphi\, d\nu_{\alpha_n}=\lambda\int_X  u\varphi\, d\nu.\]
Similarly, for every $\varphi\in\fcon^\infty_b(X)$ and $n\in\N$, by inequalities \eqref{stime max regol per alpha n} we have
\begin{gather*}
\int_X\left|\gen{\nabla_Hu_{\alpha_n},\nabla_H\varphi}_H\right| e^{-U_{\alpha_n}}d\mu
\leq
\int_X|\nabla_H u_{\alpha_n}|_H |\nabla_H\varphi|_H e^{-U_{\alpha_n}}d\mu
\leq\norm{|\nabla_H\varphi|_H}_{\infty}\int_X|\nabla_H u_{\alpha_n}|_H e^{-U_{\alpha_n}}d\mu\leq \\
\leq \norm{|\nabla_H\varphi|_H}_{\infty}\pa{\int_X|\nabla_H u_{\alpha_n}|^2_H e^{-U_{\alpha_n}}d\mu}^{\frac{1}{2}}\pa{\int_X e^{-U_{\alpha_n}}d\mu}^{\frac{1}{2}}\leq\\
\leq\frac{1}{\sqrt{\lambda}}\norm{|\nabla_H\varphi|_H}_{\infty}\norm{\ol{f}}_{\elle^2(X,\nu_{\alpha_n})}\pa{\int_X e^{-U_{1}}d\mu}^{\frac{1}{2}}\leq \frac{1}{\lambda}\norm{|\nabla_H\varphi|_H}_{\infty}\norm{\ol{f}}_{\infty}\int_X e^{-U_{1}}d\mu.
\end{gather*}
Therefore by Proposition \ref{Convergenza MY}, Proposition \ref{Integrabilita MY}, the pointwise $\mu$-a.e. convergence of $\nabla_H u_{\alpha_n}$ to $\nabla_H u$, and the Lebesgue dominated convergence theorem we get
\[\lim_{n\ra+\infty} \int_X \gen{\nabla_H u_{\alpha_n},\nabla_H\varphi}_H\, d\nu_{\alpha_n}=\int_X \gen{\nabla_H u,\nabla_H\varphi}_H\, d\nu.\]
Finally, for every $\varphi\in\fcon^\infty_b(X)$ and $n\in\N$, we get
\[\int_X\varphi\ol{f}e^{-U_{\alpha_n}}d\mu\leq\norm{\varphi}_\infty \norm{\ol{f}}_\infty\int_Xe^{-U_{1}}d\mu.\]
Thus by Proposition \ref{Convergenza MY}, Proposition \ref{Integrabilita MY} and the Lebesgue dominated convergence theorem we get
\[\lim_{n\ra+\infty} \int_X \varphi\ol{f}\, d\nu_{\alpha_n}=\int_X \varphi\ol{f}\, d\nu.\]
Taking the limit in equation \eqref{formulazione debole u_alpha_n} as $n\ra+\infty$ we get the claim.
If $f\in\elle^2(X,\nu)$, a standard density argument gives us the assertions of our theorem.
\end{proof}

\section{A charactetization of the domain of $L_\nu$: the $\nabla_H U$ $H$-Lipschitz case}\label{A charactetization of the domain of Lnu}

We recall some basic facts about the divergence operator in weighted Gaussian spaces. For every measurable map $\Phi:X\ra X$ and every $f\in\fcon^\infty_b(X)$ we define
\[\partial_{\Phi}f(x):=\lim_{t\ra 0}\frac{f(x+t\Phi(x))-f(x)}{t},\]
whenever such limit exists. If the limit exists $\mu$-a.e. in $X$ and a function $g\in \elle^1(X,\nu)$ satisfies
\begin{gather}\label{condizione divergenza}
\int_X\partial_{\Phi}fd\nu=-\int_X fgd\nu
\end{gather}
for every $f\in\fcon^\infty_b(X)$, then $g$ is called weighted Gaussian divergence of $\Phi$. Furthermore if $g$ exists, then it is unique and it will be denoted by $\diver_\nu\Phi:=g$. Finally if $\Phi\in\elle^1(X,\nu;H)$ has weighted Gaussian divergence, then equality \eqref{condizione divergenza} becomes
\[\int_X\gen{\nabla_H f,\Phi}_Hd\nu=-\int_Xf\diver_\nu\Phi d\nu\]
for every $f\in\fcon^\infty_b(X)$. We will use the following result (see \cite[Proposition 5.3]{Fer15}).

\begin{pro}\label{proposizione divergenza}
Let $U$ be a function satisfying Hypothesis \ref{ipotesi peso} such that $\nabla_H U$ is $H$-Lipschitz. Then every $\Phi\in W^{1,2}(X,\mu;H)$ has a weighted Gaussian divergence $\diver_\nu\Phi\in\elle^2(X,\nu)$ and for every $f\in W^{1,2}(X,\nu)$ the following equality holds,
\[\int_X\gen{\nabla_H f,\Phi}_H d\nu=-\int_Xf\diver_\nu\Phi d\nu.\]
Furthermore, if $\varphi_i=\gen{\Phi,e_i}_H$ for every $i\in\N$, then
\[\diver_\nu\Phi=\sum_{i=1}^{+\infty}\pa{\partial_i\varphi_i-\varphi_i\partial_i U-\varphi_i\hat{e}_i},\]
where the series converges in $\elle^2(X,\nu)$. Finally $\norm{\diver_\nu\Phi}_{\elle^2(X,\nu)}\leq\norm{\Phi}_{W^{1,2}(X,\nu;H)}$.
\end{pro}
We are now able to prove a characterization result for the domain of $L_\nu$.

\begin{thm}\label{caratterizzazione dominio}
Let $U$ be a function satisfying Hypothesis \ref{ipotesi peso} such that $\nabla_H U$ is $H$-Lipschitz. Then $D(L_\nu)= W^{2,2}(X,\nu)$. Moreover, for every $u\in D(L_\nu)$, it holds $L_\nu u=\diver_\nu\nabla_H u$ and
\begin{gather}\label{equivalenza norme}
\norm{u}_{D(L_\nu)}\leq\norm{u}_{W^{2,2}(X,\nu)}\leq\pa{2+\sqrt{2}}\norm{u}_{D(L_\nu)},
\end{gather}
where $\norm{\cdot}_{D(L_\nu)}$ is defined in formula \eqref{formula norma grafico}.
\end{thm}

\begin{proof}
Let $u\in D(L_\nu)$. We have
\(u-L_\nu u\in\elle^2(X,\nu)\). Then by Theorem \ref{Main Theorem} we get that $u\in W^{2,2}(X,\nu)$.
Let $u\in W^{2,2}(X,\nu)$, by Proposition \ref{proposizione divergenza} we get that $\diver_\nu\nabla_H u\in \elle^2(X,\nu)$ and
\[\int_X\gen{\nabla_H f,\nabla_H u}_Hd\nu=-\int_X f\diver_\nu\nabla_H ud\nu\]
for every $f\in\fcon^\infty_b(X)$. Then we have $u\in D(L_\nu)$ and $L_\nu u=\diver_\nu\nabla_H u$.

By Proposition \ref{proposizione divergenza} we have
\begin{gather}\label{equivalenza norme 1}
\norm{u}_{D(L_\nu)}=\norm{u}_{\elle^2(X,\nu)}+\norm{L_\nu u}_{\elle^2(X,\nu)}=\norm{u}_{\elle^2(X,\nu)}+\norm{\diver_\nu\nabla_H u}_{\elle^2(X,\nu)}\leq \norm{u}_{W^{2,2}(X,\nu)},
\end{gather}
for every $u\in D(L_\nu)$. Now if $u\in D(L_\nu)$, then for every $\lambda\in(0,1)$ the function $\lambda u-L_\nu u$ belongs to $\elle^2(X,\nu)$ and by Theorem \ref{Main Theorem} we get for $u\in D(L_\nu)$
\begin{gather}\label{equivalenza norme 2}
\begin{array}{c}
\displaystyle\norm{u}_{W^{2,2}(X,\nu)}\leq\pa{\frac{1}{\lambda}+\frac{1}{\sqrt{\lambda}}+\sqrt{2}}\norm{\lambda u- L_\nu u}_{\elle^2(X,\nu)}
\displaystyle \leq \pa{\frac{1}{\lambda}+\frac{1}{\sqrt{\lambda}}+\sqrt{2}}\pa{\lambda\norm{u}_{\elle^2(X,\nu)}+\norm{L_\nu u}_{\elle^2(X,\nu)}}\leq \\
\displaystyle \leq \pa{\frac{1}{\lambda}+\frac{1}{\sqrt{\lambda}}+\sqrt{2}}\pa{\norm{u}_{\elle^2(X,\nu)}+\norm{L_\nu u}_{\elle^2(X,\nu)}}=\pa{\frac{1}{\lambda}+\frac{1}{\sqrt{\lambda}}+\sqrt{2}}\norm{u}_{D(L_\nu)}.
\end{array}
\end{gather}
Letting $\lambda\ra 1^-$ in inequality \eqref{equivalenza norme 2} we get
\begin{gather}\label{equivalenza norme 3}
\norm{u}_{W^{2,2}(X,\nu)}\leq\pa{2+\sqrt{2}}\norm{u}_{D(L_\nu)}.
\end{gather}
Combining inequality \eqref{equivalenza norme 1} and inequality \eqref{equivalenza norme 3}, we get inequality \eqref{equivalenza norme}.
\end{proof}

\begin{remark}\label{remark dopo caratterizzazione dominio}
By the proof of Theorem \ref{caratterizzazione dominio}, if $U$ satisfies Hypothesis \ref{ipotesi peso} then
\begin{gather}\label{inclusione dominio}
D(L_\nu)\subseteq W^{2,2}(X,\nu),
\end{gather}
and inequality \eqref{equivalenza norme 3} holds for every $u\in D(L_\nu)$. We do not know if the additional assumption that $\nabla_H U$ is $H$-Lipschitz is necessary to guarantee the equality in formula \eqref{inclusione dominio} and inequality \eqref{equivalenza norme 1}.
\end{remark}

\section{Examples}

In this section we will denote by $d\xi$ the Lebesgue measure on $[0,1]$. We recall that a function $f:X\ra\R$ from a Banach space $X$ to $\R$ is G\^ateaux differentiable at $x\in X$ if for every $y\in X$ the limit
\begin{gather}\label{limite diff}
\lim_{t\ra 0}\frac{f(x+ty)-f(x)}{t}
\end{gather}
exists and defines a linear (in $y$) map $f'(x)(\cdot)$ which is continuous from $X$ to $\R$. Furthermore if the limit \eqref{limite diff} exists uniformly for $y\in X$ such that $\norm{y}_X=1$, then the function $f$ is said to be Fr\'echet differentiable. We will use the following result of Aronszajn (see \cite[Theorem 1 of Chapter 2]{Aro76} and \cite[Theorem 6]{Phe78}).

\begin{thm}\label{Aronszajn}
Suppose that $X$ is a separable real Banach space. If $f:X\ra\R$ is a continuous convex function, then $f$ is G\^ateaux differentiable outside of a Gaussian null set, i.e. a Borel set $A\subseteq X$ such that $\mu(A)=0$ for every non-degenerate Gaussian measure $\mu$ on $X$.
\end{thm}

Consider the classical Wiener measure $P^W$ on $\con[0,1]$ (see \cite[Example 2.3.11 and Remark 2.3.13]{Bog98} for its construction). Recall that the Cameron--Martin space $H$ is the space of the continuous functions $f$ on $[0,1]$ such that $f$ is absolutely continuous, $f'\in\elle^2([0,1],d\xi)$ and $f(0)=0$. In addition if $f,g\in H$, then $\abs{f}_H=\norm{f'}_{\elle^2([0,1],d\xi)}$ and 
\[\gen{f,g}_H=\int_0^1f'(\xi)g'(\xi)d\xi,\]
see \cite[Lemma 2.3.14]{Bog98}. An orthonormal basis of $\elle^2([0,1],d\xi)$ is given by the functions
\[e_n(\xi)=\sqrt{2}\sin\frac{\xi}{\sqrt{\lambda_n}}\qquad\text{where }\lambda_n=\frac{4}{\pi^2(2 n-1)^2}\text{ for every }n\in\N.\]
We recall that $P^W$ is a centered Gaussian measure on
\(\con_0[0,1]=\set{f\in\con[0,1]\tc f(0)=0}\),
and if $f\in H$, then 
\begin{gather}\label{norma H e norma L2}
\abs{f}^2_H=\sum_{i=1}^{+\infty}{\lambda_i}^{-1}\gen{f,e_i}^2_{\elle^2([0,1],d\xi)}.
\end{gather}

\subsection{A weight bounded from below}\label{A weight bounded from below}

Let $U(f)=\int_0^1f^2(\xi)d\xi$. Such function is convex, continuous and Fr\'echet differentiable at any $f\in\con_0[0,1]$. Moreover for every $g\in\con_0[0,1]$
\begin{gather}\label{derivata esempio 1}
U'(f)(g)=2\int_0^1f(\xi)g(\xi)d\xi.
\end{gather}
By the fact that $U(f)\leq\norm{f}^2_\infty$ and the Fernique theorem \cite[Theorem 2.8.5]{Bog98}, we get that $U\in \elle^t(\con_0[0,1],P^W)$ for every $t\geq 1$. By formula \eqref{derivata esempio 1} we get
\[\nabla_H U(f)=2\sum_{i=1}^{+\infty}\pa{\int_0^1f(\xi)e_i(\xi)d\xi}e_i.\]

By \cite[Proposition 5.4.6]{Bog98} if we show that $\nabla_H U$ is integrable for every $t\geq 1$ we get that $U\in W^{1,t}(\con_0[0,1],P^W)$. We claim that $\nabla_H U$ is $H$-Lipschitz. Indeed, for $f\in\con_0[0,1]$ and $h\in H$, we have
\begin{gather*}
\abs{\nabla_H U(f+h)-\nabla_H U(f)}^2_H=4\sum_{i=1}^{+\infty}\abs{\partial_i U(f+h)-\partial_i U(f)}^2=\\
=4\sum_{i=1}^{+\infty}\abs{\int_0^1 (f(\xi)+h(\xi))e_i(\xi)d\xi-\int_0^1 h(\xi)e_i(\xi)d\xi}^2=4\sum_{i=1}^{+\infty}\abs{\int_0^1 h(\xi)e_i(\xi)d\xi}^2=4\sum_{i=1}^{+\infty}\gen{h,e_i}_{\elle^2([0,1],d\xi)}^2.
\end{gather*}
Observing that $\lambda_i\leq 4\pi^{-2}$ for every $i\in\N$, we get
\begin{gather*}
\abs{\nabla_H U(f+h)-\nabla_H U(f)}^2_H= 4\sum_{i=1}^{+\infty}\gen{h,e_i}_{\elle^2([0,1],d\xi)}^2= 4\sum_{i=1}^{+\infty}\frac{\lambda_i}{\lambda_i}\gen{h,e_i}_{\elle^2([0,1],d\xi)}^2\leq\\
\leq \frac{16}{\pi^2}\sum_{i=1}^{+\infty}\lambda_i^{-1}\gen{h,e_i}_{\elle^2([0,1],d\xi)}^2=\frac{16}{\pi^2}\abs{h}^2_H
\end{gather*}
where the last equality follows from formula \eqref{norma H e norma L2}. According to \cite[Theorem 5.11.2]{Bog98} $\nabla_H U$ belongs to $W^{1,t}(\con_0[0,1],P^W;H)$ for every $t>1$. Thus $U$ satisfies Hypothesis \ref{ipotesi peso 2}.

Theorem \ref{Stime per lip} can be applied in this case. So for every $f\in\elle^2(\con_0[0,1],e^{-U}P^W)$ and $\lambda>0$ there exists a unique weak solution (which is also a strong solution, in the sense of Definition \ref{definizione strong solution}) $u\in W^{2,2}(\con_0[0,1],e^{-U}P^W)$ of equation \eqref{Problema} and $u$ satisfies the Sobolev regularity estimates of Theorem \ref{Main Theorem}. Furthermore we can apply Theorem \ref{caratterizzazione dominio} and get
\[D(L_{e^{-U}P^W})=W^{2,2}(\con_0[0,1],e^{-U}P^W).\]
Finally, for every $v\in D(L_{e^{-U}P^W})$, we have $L_{e^{-U}P^W} v=\diver_{e^{-U}P^W}\nabla_H u$ and
\begin{gather*}
\norm{v}_{D(L_{e^{-U}P^W})}\leq\norm{v}_{W^{2,2}(\con_0[0,1],e^{-U}P^W)}\leq\pa{2+\sqrt{2}}\norm{v}_{D(L_{e^{-U}P^W})},
\end{gather*}
where $\norm{\cdot}_{D(L_{e^{-U}P^W})}$ is defined in formula \eqref{formula norma grafico}.

\subsection{A unbounded weight}\label{A unbuonded weight}

Let $F(f)=\max_{\xi\in[0,1]}f(\xi)$ for $f\in\con[0,1]$. In order to compute the G\^ateaux derivative of $F$ we will use some classical arguments.

\begin{pro}\label{Proprieta massimo}
$F$ is Lipschitz continuous, convex and G\^ateaux differentiable at $f\in\con_0[0,1]$ if, and only if, $f\in M$ where
\[M=\set{f\in\con_0[0,1]\tc \text{ there exists a unique $\xi_f\in[0,1]$ such that }F(f)=f(\xi_f)}.\]
Furthermore, if $f\in M$ and $\xi_f$ is the unique maximum point of $f$, then $F'(f)(g)=g(\xi_f)$.
\end{pro}

\begin{proof}
Convexity and Lipschitz continuity are obvious. Let $\xi_f$ be the unique maximum point of $f\in M$ and for $t\in\R$ and $g\in\con_0[0,1]$ choose $\xi_t\in[0,1]$ such that \(F(f+tg)=f(\xi_t)+tg(\xi_t)\). Observe that
\begin{gather*}
0\leq f(\xi_f)-f(\xi_t)= f(\xi_f)+tg(\xi_f)-tg(\xi_f)-f(\xi_t)\leq f(\xi_t)+tg(\xi_t)-tg(\xi_f)-f(\xi_t)=t(g(\xi_t)-g(\xi_f)).
\end{gather*}
So we have 
\[\abs{f(\xi_f)-f(\xi_t)}\leq t(g(\xi_t)-g(\xi_f))\leq 2\abs{t}\norm{g}_\infty.\] 
Since $f\in M$ we have that $\xi_t\ra\xi_f$ for $t\ra 0$.
Observe that
\begin{gather}\label{differentiabilita max 1}
F(f+tg)-F(f)-tg(\xi_f)\geq f(\xi_f)+tg(\xi_f)-f(\xi_f)-tg(\xi_f)=0;\\
\intertext{and}\label{differentiabilita max 2}
F(f+tg)-F(f)-tg(\xi_f)\leq f(\xi_t)+tg(\xi_t)-f(\xi_t)-tg(\xi_f)
\leq t(g(\xi_t)-g(\xi_f))\leq \abs{t}\abs{g(\xi_t)-g(\xi_f)}.
\end{gather}
By inequality \eqref{differentiabilita max 1} and inequality \eqref{differentiabilita max 2} we have that if $f\in M$, then $F$ is G\^ateaux differentiable at $f$ and $F'(f)(g)=g(\xi_f)$.

Assume now that $F$ is G\^ateaux differentiable at $f\in\con_0[0,1]\ssm M$. Let $\xi_1,\xi_2\in[0,1]$ such that $F(f)=f(\xi_1)=f(\xi_2)$ and $\xi_1\neq \xi_2$. Set $g_1(\xi)=\abs{\xi-\xi_1}$ and observe that
\begin{gather*}
F(f+tg_1)-F(f)\geq f(\xi_2)+tg_1(\xi_2)-f(\xi_2)=t\abs{\xi_2-\xi_1};\\
F(f+tg_1)-F(f)\geq f(\xi_1)+tg_1(\xi_1)-f(\xi_1)=0. 
\end{gather*}
These inequalities gives us the following contradiction:
\begin{gather*}
\limsup_{t\ra 0^-}\frac{F(f+tg_1)-F(f)}{t}\leq 0,\qquad \liminf_{t\ra 0^+}\frac{F(f+tg_1)-F(f)}{t}\geq\abs{\xi_2-\xi_1}>0.
\end{gather*}
\end{proof}

Let $\delta_1$ be the Dirac measure concentrated in $1$ and consider the weight $e^{-U}$ where, for $f\in\con_0[0,1]$,
\[U(f)=F(f)+\delta_1(f).\]
By Proposition \ref{Proprieta massimo}, $U$ is Lipschitz continuous and convex, and it is easy to show that $U$ is unbounded, both from above and from below.

According to \cite[Theorem 5.11.2]{Bog98} $U\in W^{1,t}(\con_0[0,1],P^W)$ for every $t>1$. So $U$ satisfies Hypothesis \ref{ipotesi peso}. Furthermore by \cite[Definition 5.2.3 and Proposition 5.4.6(iii)]{Bog98}, Theorem \ref{Aronszajn} and by Proposition \cite[Proposition 4.6]{Fer15} it can be seen that
\begin{gather*}
\nabla_H U(f)=\sum_{i=1}^{+\infty}(e_i(\xi_f)+e_i(1))e_i\qquad\text{for }P^W\text{-a.e. } f\in\con_0[0,1].
\end{gather*}

Theorem \ref{Main Theorem} can be applied in this case. So for every $f\in\elle^2(\con_0[0,1],e^{-U}P^W)$ and $\lambda >0$ there exists a unique weak solution $u\in W^{2,2}(\con_0[0,1],e^{-U}P^W)$ of equation \eqref{Problema}. In addition
\begin{gather*}
\norm{u}_{\elle^2(\con_0[0,1],e^{-U}P^W)}
\leq\frac{1}{\lambda}\norm{f}_{\elle^2(\con_0[0,1],e^{-U}P^W)};\quad
\norm{\nabla_H u}_{\elle^2(\con_0[0,1],e^{-U}P^W;H)}\leq\frac{1}{\sqrt{\lambda}}\norm{f}_{\elle^2(\con_0[0,1],e^{-U}P^W)};\\
\norm{\nabla_H^2 u}_{\elle^2(\con_0[0,1],e^{-U}P^W;\mathcal{H}_2)}\leq \sqrt{2}\norm{f}_{\elle^2(\con_0[0,1],e^{-U}P^W)}.
\end{gather*}
Moreover by Remark \ref{remark dopo caratterizzazione dominio}, we obtain that \(D(L_{e^{-U}P^W})\subseteq W^{2,2}(\con_0[0,1],e^{-U}P^W)\), and for every $v\in D(L_{e^{-U}P^W})$ we have
\[\norm{v}_{W^{2,2}(\con_0[0,1],{e^{-U}P^W})}\leq\pa{2+\sqrt{2}}\norm{v}_{D(L_{e^{-U}P^W})},\] 
where $\norm{\cdot}_{D(L_{e^{-U}P^W})}$ is defined in formula \eqref{formula norma grafico}.

\bibliographystyle{plain}
\nocite{*} 
\bibliography{bibpesopoincare}

\end{document}